\documentclass[letterpaper,11pt,reqno]{amsart}
\usepackage{amssymb}
\usepackage{amscd}
\usepackage{amsfonts}
\usepackage{amsmath}
\usepackage{amsthm}
\usepackage[all]{xy}
\usepackage{diagrams}
\usepackage{mathrsfs}
\usepackage[normalem]{ulem}
\usepackage{tikz}
\usepackage{etex}
\usepackage[paper = letterpaper, left = 30mm, right = 30mm, headsep = 7mm,
footskip = 10mm, top = 30mm, bottom = 30mm, footnotesep=5mm, headheight =
2cm]{geometry}

\usepackage{titletoc}
\titlecontents{part}[0pt]{}{\bfseries\chaptername\ \thecontentslabel\quad}{\bfseries}{\bfseries\hfill}
\usepackage{xr}
\newcommand{\Iref}[2]{\cite[ #1 \ref{GRTI-#2}]{GRTI}}

\usepackage[colorlinks=true]{hyperref}
\usepackage{url}
\usepackage{cite}
\usepackage{xcolor}
\renewcommand\thepart{\Roman{part}}

\def\Amp{\mathrm{Amp}}
\def\Cone{\mathrm{Cone}^*}

\def\rank{\mathrm{rank}}

\def\rk{\mathrm{rank}}
\def\Todd{\mathrm{Todd}}

\def\ch{\mathrm{ch}}

\def\part#1{%
  \vskip .02\vsize 
  \refstepcounter{part}
  \addcontentsline{toc}{part}{Part \thepart:\ #1}
  {\centering\large \textbf{Part \thepart}. #1\par}%
  \vskip .01\vsize
}

\newcommand{\Hom}{\operatorname{Hom}}
\newcommand{\Pos}{\operatorname{Pos}}
\newcommand{\Hilb}{\operatorname{Hilb}}
\newcommand{\hilb}{\operatorname{hilb}}

\newcommand{\vol}{\operatorname{vol}}

\newcommand*{\leftrightdash}[1][]{\mathbin{\tikz [baseline=-0.25ex,-latex, densely dashed,#1] \draw [<->] (0pt,0.5ex) -- (1.3em,0.5ex);}}

\theoremstyle{plain}
\newtheorem{theorem}{Theorem}[section]
\newtheorem*{theorem*}{Theorem}
\newtheorem{lemma}[theorem]{Lemma}

\newtheorem*{conjecture*}{Conjecture}

\newtheorem{corollary}[theorem]{Corollary}
\newtheorem*{corollary*}{Corollary}

\newtheorem{proposition}[theorem]{Proposition}
\theoremstyle{definition}
\newtheorem*{definition*}{Definition}

\numberwithin{equation}{section} 

\makeatletter
\def\ifdraft{\ifdim\overfullrule>\z@
  \expandafter\@firstoftwo\else\expandafter\@secondoftwo\fi}
\makeatother

\newsavebox{\ieeealgbox}

\newcommand{\N}{\mathbb{N}}

\newcommand{\Q}{\mathbb{Q}}
\newcommand{\R}{\mathbb{R}}
\newcommand{\Z}{\mathbb{Z}}

\theoremstyle{definition}
\newtheorem{definition}[theorem]{Definition}
\newtheorem{remark}[theorem]{Remark}
\newtheorem*{remark*}{Remark}

\newtheorem{comment2}[theorem]{Comment}

\title[Semi-continuity of Stability for Sheaves ]{Semi-continuity of Stability for Sheaves and \\ Variation of Gieseker Moduli Spaces}
\date{\today}
\author[D.~Greb]{Daniel Greb}
  \address{DG: Essener Seminar f\"ur Algebraische Geometrie und Arithmetik, Fakult\"at f\"ur Mathematik, Universit\"at Duisburg-Essen, 45117 Essen, Germany}
  \email{daniel.greb@uni-due.de}
\author[J.~Ross]{Julius Ross}
  \address{JR: Department of Pure Mathematics and Mathematical Statistics, University of Cambridge, Wilberforce Road, Cambridge, CB3 0WB, UK}
  \email{j.ross@dpmms.cam.ac.uk}
\author[M.~Toma]{Matei Toma}
\address{MT: Institut de Math\'ematiques \'Elie Cartan, Universit\'e de Lorraine, B.P. 70239, \newline 54506 Vandoeuvre-l\`es-Nancy Cedex,
France}
\email{Matei.Toma@univ-lorraine.fr}
\thanks{}
\begin{document}
\vspace{-0.2cm}

\begin{abstract}
We investigate a semi-continuity property for stability conditions for sheaves that is important for the problem of variation of the moduli spaces as the stability condition changes.  We place this in the context of a notion of stability previously considered by the authors, called multi-Gieseker-stability, that generalises the classical notion of Gieseker-stability to allow for several polarisations.  As such we are able to prove that on smooth threefolds certain moduli spaces of Gieseker-stable sheaves are related by a finite number of Thaddeus-flips (that is flips arising for Variation of Geometric Invariant Theory) whose intermediate spaces are themselves moduli spaces of sheaves.
\end{abstract}
\maketitle

\tableofcontents
\newcommand{\Space}{\,\,\,\,\,\,\,\,\,\,\,\,\,\,\,\,}

\addtocontents{toc}{\protect\setcounter{tocdepth}{0}}

\section*{Introduction}
The construction of moduli spaces of sheaves usually requires a choice of a stability condition, and in most cases different choices yield different spaces. The purpose of this paper is to  investigate a ``semi-continuity" property for such stability conditions which, as we shall see, plays an important role in the variation of these  moduli spaces.

To begin the discussion, we somewhat vaguely denote a choice of stability condition by $\sigma$ which is allowed to vary in some space $\Sigma$.  The best-known examples are slope-stability  and Gieseker-stability, in which case $\sigma$ is a choice of (rational) ample line bundle and so $\Sigma$ is the ample cone of the base.   More interesting examples arise from the stability conditions of  Joyce \cite{JoyceII}, that include the (twisted) multi-Gieseker stability previously considered by the authors \cite{GRTI}.

To discuss the variation problem, consider a path $\sigma(t)\in \Sigma$ for $t\in [0,1]$ that has a single ``critical value" at the point $t=\overline{t}$.  By this we mean that stability is unchanged as $t$ varies within either the interval $[0,\overline{t})$ or the interval $(\overline{t},1]$.  Thus we are left with three moduli spaces that we denote by   $\mathcal M_{\sigma(0)}$, $\mathcal M_{\sigma(\overline{t})}$ and $\mathcal M_{\sigma(1)}$ and we wish to understand how these are related.  In the best possible case, there would exist a diagram of the form
\begin{equation}\label{eq:flipintro}\tag{$\ast$}
\begin{gathered}
\begin{xymatrix}{
\mathcal M_{\sigma(0)} \ar[rd]& & \ar[ld] \mathcal M_{\sigma(1)}\\
& \mathcal M_{\sigma(\overline{t})},&
}
\end{xymatrix}
\end{gathered}\end{equation}
where the morphisms are induced from functorial properties of these spaces.     For this to be possible in higher dimension, we immediately face two problems:

\begin{enumerate}
\item (Existence of Moduli at Critical Values) We need that the moduli space $\mathcal M_{\sigma(\overline{t})}$ exists; that is we need to be able to construct the moduli spaces for critical values in $\Sigma$.  
\item (Semicontinuity) Assuming (1)  holds, the existence of a diagram of the form \eqref{eq:flipintro} requires that every sheaf that is semistable with respect to $\sigma(0)$ (resp.\ $\sigma(1)$) must also be semistable for $\sigma(\overline{t})$.  
\end{enumerate}

The existence of the moduli spaces at critical values can be problematic (on manifolds of dimension at least 3) for both slope-stability and Gieseker-stability, as the critical point in the ample cone at which stability changes may be irrational (see the discussion in \cite{GRTI}).    Furthermore semi-continuity will generally fail for stability conditions derived from a polynomial inequality (such as Gieseker-stability, see Remark \ref{rmk:closedness})) that are natural from the viewpoint of algebraic geometry.
 We will explain towards the end of the introduction why a naive approach based on the theory of Variation of GIT as developed in \cite{Thaddeus} and \cite{DolgachevHu} fails to solve these problems.

In this paper we study a notion of stability for sheaves which is restrictive enough to be able to produce moduli spaces at critical values (that are even projective) yet flexible enough to ensure that one can (conjecturally) always arrange so this semi-continuity holds.   As will be discussed presently, we prove this conjecture in some particular cases, yielding an application to the variation problem for the moduli spaces of Gieseker stable sheaves. 



\subsection*{Twisted Multi-Gieseker-stability}

In previous work of the authors \cite{GRTI} we introduce a notion of ``multi-Gieseker stability" that extends the usual notion of Gieseker-stability of sheaves to the case of several polarisations, allowing us to study the variation problem in this context.   For the purpose of this paper we require a slight generalisation that permits a fixed twisting of the Hilbert polynomials involved.  To define this precisely, suppose that $L_1,\ldots, L_{j_0}$ are fixed ample line bundles on a smooth projective manifold $X$ and that $B_1,\ldots,B_{j_0}$ are a further collection of fixed line bundles. Given $\sigma_1,\ldots,\sigma_{j_0}\in \mathbb R_{\ge 0}$, not all zero, we shall say a torsion-free coherent sheaf $E$ on $X$ is \emph{semistable} with respect to this data if for all proper subsheaves $E'\subset E$ the inequality 
\begin{equation*}
 \frac{\sum_{j} \sigma_j \chi(E'\otimes L_{j}^{k}\otimes B_{j})}{\rank(E')} \le  \frac{\sum_{j} \sigma_j \chi(E\otimes L_{j}^{k}\otimes B_{j})}{\rank(E)}\label{eq:introstability}
 \end{equation*}
holds for all $k$ sufficiently large.    We refer to $\sigma=(\sigma_1,\ldots,\sigma_{j_0})$ as a \emph{stability parameter} which we say is \emph{bounded} if the set of semistable sheaves of a given topological type is bounded.

This definition generalises that of \cite{GRTI} in which all the $B_j$ were taken to be trivial, and introduces no new difficulties so all the results from \cite{GRTI} extend to this setting (see Section \ref{sec:twisted} for further discussion).  In particular, semistable sheaves have projective moduli:

\begin{theorem*}[Existence of projective moduli spaces of twisted semistable sheaves]
Suppose $\sigma$ is a bounded stability parameter.    Then, there exists a projective coarse moduli space $\mathcal M_{\sigma}$ of semistable sheaves of a given topological type.   The moduli spaces $\mathcal M_{\sigma}$ contain an open set parameterizing stable sheaves, and the points on the boundary correspond to $S$-equivalence classes of sheaves.  
\end{theorem*}

A noteworthy aspect of this construction is that the critical values (or ``walls") in $\Sigma: = (\mathbb R_{>0})^{j_0}\setminus \{0\}$ that witness the change of stability are given by rational linear functions.  Thus there is no loss in assuming that all the $\sigma_j$ are rational, from which we conclude that the desired moduli spaces exist even at the critical values.

 Hence we turn our attention to semi-continuity.  Fixing the $L_j$ and $B_j$ as above, we say that a function $\sigma(t) = (\sigma_1(t),\ldots,\sigma_{j_0}(t)) \in \mathbb R_{\ge 0}^{j_0}\setminus \{0\}$ for $t\in [0,1]$ is a \emph{stability segment} if each $\sigma_j$ is linear and
 $$\sum\nolimits_j \vol(L_j) \sigma_j(t) = 1 \quad \text{ for all } t \in [0,1]$$
 where $\vol(L_j) : = \int_X c_1(L_j)^d$ and $d=\dim X$.
 
\begin{definition*}[Uniform stability]
We say a stability segment $(\sigma(t))_{t\in [0,1]}$  is \emph{uniform} if for any torsion-free sheaf $E$ we have
$$  \frac{\sum_{j} \sigma_j \chi(E\otimes L_{j}^{k}\otimes B_{j})}{\rank(E)}=  \frac{k^d}{d!} + a_{d-1}(E) k^{d-1} + \cdots + a_1(E)k + a_0(E,t) \text{ for }t\in [0,1]$$
where  $a_{d-1}(E),\ldots,a_{1}(E)$ are independent of $t$ and $a_0(E,t)$ is linear in $t$.

\end{definition*}

It is not hard to see that a uniform stability segment has the semi-continuity property (Remark \ref{rmk:closedness}).   Roughly speaking then, the two main results of this paper say that uniform stability is (a) sufficiently strong to be able to control the variation of the moduli spaces as $t$ varies and (b) sufficiently flexible so that such segments can actually be constructed.  \medskip

We discuss each of these in turn.  For the first, by point of terminology, we say a  stability segment $(\sigma(t))_{t\in [0,1]}$ is \emph{bounded} if the set of sheaves of a given topological type that are semistable with respect to $\sigma(t)$ for some $t\in [0,1]$ is bounded. Throughout the paper by the topological type of a sheaf $E$ we shall mean its homological Todd class $\tau_X(E)\in B(X)_\Q:=B(X)\otimes_\Z\Q$, where $B(X)$ is the group of cycles on $X$ modulo algebraic equivalence, see \cite[Def.~1.4]{GRTI}.

\begin{theorem*}(Thaddeus-flips through moduli spaces of sheaves I, Theorem \ref{thm:intermediate})
Let $X$ be smooth and projective, let $\tau\in B(X)_{\mathbb Q}$ and $(\sigma(t))_{t\in [0,1]}$ be a bounded uniform stability segment.  Then given any $t'<t''$ in $(0,1)$\footnote{The assumption that $t'$ and $t''$ lie in the open-interval $(0,1)$ will be relaxed in a sequel to include the endpoints.} the spaces $\mathcal M_{\sigma(t')}$ and $\mathcal M_{\sigma(t'')}$ are connected by a finite collection of Thaddeus-flips of the form
\[\begin{xymatrix}{
 \mathcal{M}_{\sigma(t_i)} \ar[rd] &  & \mathcal{M}_{\sigma(t_{i+1})} \ar[ld] \\
                 &        \mathcal{M}_{\sigma(t'_i)}.               & }
\end{xymatrix}
\]

\end{theorem*}

Thus uniform stability segments connect moduli spaces by Thaddeus-flips \emph{through moduli spaces of sheaves}.  Having established this result we will turn to the problem of constructing uniform stability segments.  We show that on surfaces and threefolds this is possible, allowing us to conclude the following:

\begin{theorem*}[Existence of Uniform Stability Segments, Theorems \ref{thm:2intermediate} and \ref{thm:3intermediate}]
Suppose $X$ is smooth and projective, let $\tau\in B(X)_{\mathbb Q}$, and $L',L''\in \Amp(X)$. Suppose also that either
\begin{enumerate}
\item $X$ has dimension 2 or
\item $X$ has dimension 3 and $L',L''$ are separated by a single wall of the first kind.
\end{enumerate}
 Then, the moduli spaces $\mathcal M_{L'}$ and $\mathcal M_{L''}$ of Gieseker-semistable torsion-free sheaves of topological type $\tau$ are related by a finite number of Thaddeus-flips through spaces of the form $\mathcal M_{\sigma_i}$ for some bounded stability parameters $\sigma_i$.\end{theorem*}

The conclusion of this theorem for surfaces is precisely that of Matsuki-Wentworth \cite{MatsukiWentworth}, and our real interest lies in higher dimensions.    We refer the reader to Definition~ \ref{def:separatedwallfirstkind} for the precise definition of being ``separated by a single wall of the first kind'', but we do not expect the hypotheses of this theorem to be optimal or particularly important.  As will become apparent what we really discuss is a framework for producing uniform stability segments in general and we expect the following to be true:

\begin{conjecture*}
Any two bounded stability parameters $\sigma'$ and $\sigma''$ can be joined by a bounded uniform stability segment,
\end{conjecture*}

If true one would conclude that any two moduli spaces $\mathcal M_{\sigma'}$ and $\mathcal M_{\sigma''}$ are related by Thaddeus-flips through other moduli spaces of sheaves (in particular this would hold for any two moduli spaces of Gieseker semistable sheaves taken with respect to different polarisations).   It seems likely that the above conjecture can be solved using our framework,  but the details become much more involved.

\subsection*{Zooming}

We briefly discuss the main idea used to construct uniform stability segments (see also Section \ref{sec:strategy} for a more detailed summary).   Suppose that we have two bounded stability parameters $\sigma'$ and $\sigma''$ (taken with respect to the same line bundles $L_{j},B_{j}$) and we wish to show that the moduli spaces $\mathcal M_{\sigma'}$ and $\mathcal M_{\sigma''}$ are related by Thaddeus-flips through moduli spaces of sheaves.  The natural thing to do is to consider the straight line segment $(\sigma(t))_{t\in [0,1]}$ of stability parameters that joins $\sigma'$ and $\sigma''$.  If this segment were uniform, then we would be done by the discussion above, but in general this is not the case.  Instead, there will be a finite number of rational ``walls'' that witness the change of multi-Gieseker-stability with respect to $\sigma(t)$ as $t$ varies in $[0,1]$.  For simplicity, we assume these walls all lie in the interior $(0,1)$.  

Now, the idea is to ``zoom in'' on a single one of these walls $\overline{t}$, and try to find another stability segment $(\eta(s))_{s\in [0,1]}$ for which the following holds:

\medskip
\noindent Let $t_{-}$ and $t_{+}$ be points in $(0,1)$ immediately to the left and right of $\overline{t}$.   Then, for $s>0$ sufficient small we have $\mathcal M_{\sigma(t_{-})} = \mathcal M_{\eta(s)} \text{ and }\mathcal M_{\sigma(t_{+})} = \mathcal M_{\eta(1-s)}$.\medskip

Now, if $\eta$ were uniform, then by the discussion above (Theorem \ref{thm:intermediate}) we get that $\mathcal M_{\sigma(t_{-})}$ and $\mathcal M_{\sigma(t_{+})}$ are related by Thaddeus-flips though moduli spaces of sheaves.  Thus, by applying this idea across each of the finitely many walls $\overline{t}$, we deduce the same holds for $\mathcal M_{\sigma'}$ and $\mathcal M_{\sigma''}$.

It seems unlikely that one can always construct such an $\eta$ that is uniform.  However, for surfaces and threefolds at least, we can construct such a segment that is ``closer'' to being uniform than $\sigma$ was.  Roughly speaking, by this we mean that more coefficients of the relevant reduced multi-Hilbert polynomials associated with $\eta(s)$ are independent of $s$.  
We can then apply the same argument, and zoom-in on the walls for $\eta(s)$, and repeat this until eventually we are left with a  chain of uniform stability segments, which join $\sigma'$ and $\sigma''$.

\subsection*{Comparison with a Naive Approach}
The moduli spaces of Gieseker-semistable sheaves (and more generally twisted multi-Gieseker-semistable sheaves)  are constructed using Geometric Invariant Theory (GIT) and thus one may naively expect that the existence of diagrams of the form \eqref{eq:flipintro}
follows trivially from the theory of Variation of GIT due to Thaddeus \cite{Thaddeus} and Dolgachev-Hu \cite{DolgachevHu}.   However there are a number of subtleties involved in making this true that are closely connected to the semi-continuity property discussed above.

To discuss this further, consider first the case of the GIT construction of the moduli space of sheaves that are Gieseker-semistable with respect to an ample line bundle $L_0$.  The first task is to find a space with a group action, whose orbits correspond to equivalence classes semistable sheaves.  To this end one argues as follows: the set of sheaves $E$ that are Gieseker-semistable with respect to $L_0$ and of a given topological type is bounded, and thus for $n$ sufficiently large can be considered as a quotient
$$H^0(E\otimes L_0^n) \otimes L_0^{-n} \to E\to 0.$$
Picking an identification between $H^0(E\otimes L_0^n)$ and a fixed vector space of the appropriate dimension, we can thus think of $E$ as a point in a certain Quot scheme $Q_{L_0}$.   This space has a group action, coming from the different choices of this identification, and the tools of GIT are then applied to an appropriate subscheme $R_{L_0}\subset Q_{L_0}$ to get the desired moduli space $\mathcal M_{L_0}$.  This requires a choice of ample polarisation on $R_{L_0}$, and it is a theorem \cite[Thm. 4.3.3]{Bible} that, for $n$ sufficiently large,  there is a choice of polarisation on $R_{L_0}$ for which GIT stability agrees with Gieseker-stability.

From this description it is clear that there is a problem in using Variation of GIT when $L_0$ is replaced with some other ample line bundle $L_1$.  For the moduli space $\mathcal M_{L_1}$ is constructed as the quotient \emph{of a different space} $R_{L_1}$ and not the space $R_{L_0}$ with respect to some polarisation.   (It would be nice if one could also produce $\mathcal M_{L_1}$ as a quotient of $R_{L_0}$ with respect to some suitable polarisation,  but it is not clear if this is in fact possible.)

A slightly different issue occurs when considering sheaves that are (twisted) multi-Gieseker-stable.  Here the ample line bundles $L_1,\dots L_{j_0}$ and twistings $B_1,\ldots,B_{j_0}$ are fixed, and the stability parameter $\sigma\in \mathbb R^{j_0}\setminus \{0\}$ is allowed to vary.   In this context,  the parameter space that replaces $R_{L_0}$ in the previous paragraph is a representation space of a certain quiver that depends on two integers $m\gg n\gg 0$.    The subtlety now concerns the matching of stability in the sense of sheaves with stability in the sense of GIT.  In \cite{GRTI} we prove that for $m\gg n\gg 0$ these two notions agree for a \emph{fixed} $\sigma$.  In this paper we extend this by showing that under the assumption that $\sigma$ varies within a \emph{uniform} segement of stability parameters   these integers can be chosen uniformly (thus justifying the terminology).  As a conclusion we get the variation of stability in the sense of sheaf theory agrees with the variation of stability in the sense of GIT, and thus conclude that the moduli spaces are related by Thaddeus-flips through moduli spaces of sheaves.

\subsection*{Gieseker-stability with respect to real ample classes}

As a by-product of the material we develop to define what it means for a polarisation to be general, cf.~Section~\ref{sec:variationgeneral}, we can give a statement about the existence and projectivity of moduli spaces of sheaves that are Gieseker-semistable with respect certain real classes $\omega\in \Amp(X)_{\R}$.  We refer the reader to Section \ref{sec:chambrstructuregieseker} for definitions of the walls $\tilde{W}_{i,F}$ used in the following statement.

\begin{theorem*}[Projective moduli spaces for $\omega$-semistable sheaves, Theorem \ref{thm:kaehlermoduli:revisited}]
 Let $X$ be a smooth projective manifold of dimension $d$ and $\tau \in B(X)_{\mathbb Q}$.  Suppose that $K\subset \Amp(X)_{\mathbb R}$ is open and relatively compact and that $\omega\in \Cone(K)$ does not lie on any of the walls $\tilde{W}_{i,F}$ for $i\ge 2$ and $F\in \mathcal S_K$. 

Then, there exists a projective moduli space $\mathcal M_{\omega}$ of torsion-free sheaves of topological type $\tau$ that are Gieseker-semistable with respect to $\omega$.  This moduli space contains an open set consisting of points representing isomorphism classes of stable sheaves, while points on the boundary correspond to $S$-equivalence classes of strictly semistable sheaves.  
 \end{theorem*}

We emphasise that the novelty of the previous Theorem is that $\omega$ is allowed to be real, possibly irrational, and we recall that for threefolds we proved more \Iref{Theorem}{thm:kaehlermodulithreefolds}, namely that $\mathcal M_{\omega}$ is projective for all $\omega\in \Amp(X)_{\mathbb R}$.  It would be interesting to know if this continues to hold for manifolds of any dimension.

\subsection*{Acknowledgements } The authors wish to thank Arend Bayer for helpful conversations at a crucial stage of this project. We also thank Dominic Joyce, Jun Li, and Alexander Schmitt for discussions comparing this work to theirs. In addition, we wish to thank Ivan Smith for discussions, and also acknowledge inspiration drawn from a talk by Aaron Bertram given in January 2014 at Banff  International Research Station \cite{Bertramb}. \medskip

\addtocontents{toc}{\protect\setcounter{tocdepth}{1}}

\part{Uniformity}

We start by extending the discussion of \cite{GRTI} in two directions, first to allow a twisting of the stability condition by a collection of fixed line bundles, and second to make our main results uniform as the stability parameter varies.

\section{Twisted multi-Gieseker-stability}\label{sec:twisted}

Let $X$ be a projective manifold of dimension $d$.  Suppose in addition to a vector $\underline{L}=(L_1,\ldots,L_{j_0})$ of ample line bundles we fix a vector $\underline{B} = (B_1,\ldots, B_{j_0})$ of line bundles on $X$.  Given a collection $\sigma_{1},\ldots,\sigma_{j_{0}}$ of non-negative real numbers, not all zero, we define the \emph{multi-Hilbert polynomial} with respect to this data of a torsion-free sheaf $E$ to be
$$P_E^{\sigma} (m) = \sum\nolimits_j \sigma_j \chi(E\otimes L^m_j\otimes B_j).$$
Writing this polynomial as $P_E^{\sigma}(m) =\sum_{i=0}^d \alpha^{\sigma}_i(E) \frac{m^i}{i!},$ we define the \emph{reduced multi-Hilbert polynomial} as
$$ p_{E}^{\sigma} = \frac{P_{E}^{\sigma}}{\alpha^{\sigma}_{d}(E)}.$$
We shall refer to either the data $(\underline{L},\underline{B},\sigma_{1},\ldots,\sigma_{j_{0}})$ or the vector $(\sigma_{1},\ldots,\sigma_{j_{0}})$ as a \emph{stability parameter}.

\begin{definition}[Multi-Gieseker-stability]\label{defi:multiGiesekerstability}
   We say that a torsion-free sheaf $E$ is \emph{multi-Gieseker-(semi)stable}, or just \emph{(semi)stable},  if for all proper subsheaves $0 \neq F\subset E$ we have 
  \begin{equation}
p_F^{\sigma} (\le) p_E^{\sigma}.
\end{equation}
\end{definition}
Just as in the case of Gieseker-stability, any semistable sheaf admits a J\"ordan-Holder filtration, whose graded object is denoted $gr(E)$.  We say two semistable sheaves $E$ and $E'$ are \emph{$S$-equivalent} if $gr(E)$ and $gr(E')$ are isomorphic.

With these definitions at hand, all the results from \cite{GRTI} extend to this twisted setting.  More precisely, in this more general context the analogues of the Embedding Theorem, \Iref{Theorem}{thm:categoryembedding} (using the same quiver $Q$), and of the main ``Comparison of semistability and JH filtrations'' result, \Iref{Theorem}{thm:mainsemistabilitycomparison}, hold. As a corollary of these results, there exist projective moduli spaces parametrising $S$-equivalence classes of $\sigma$-semistable sheaves of fixed topological type $\tau$, which we will denote by $\mathcal{M}_\sigma = \mathcal{M}_\sigma(\tau)$.

We do not repeat the proofs in this twisted setting, as no new ideas are needed and the existing proofs go through with essentially trivial modifications (and moreover many parts will actually follow from the work in the next section).  The main change is that we now define
$$ T  :=\bigoplus\nolimits_{ij} L_i^{-n} \otimes B_i^{-1} \oplus L_j^{-m} \otimes B_j^{-1},$$
so the quiver representation associated with a sheaf $E$ becomes
$$ \Hom(T,E):= \bigoplus\nolimits_{ij} H^0(E\otimes L_i^n \otimes B_i) \oplus H^0(E\otimes L_j^m \otimes B_j).$$
This is an $L\oplus H$ algebra, where $L$ is the algebra generated by the projection operators (precisely as in \Iref{Section}{subsubsect:Qreps}) and where we adjust the definition of $H$ to be
$$ H := \bigoplus\nolimits_{ij} H_{ij} := \bigoplus\nolimits_{ij} H^0(X,L_i^{-n}\otimes L_j^m \otimes B_i^{-1} \otimes B_j).$$
To obtain the necessary regularity of the sheaves involved, we replace any instance of $(r,\underline{L})$-regularity with $(r,\underline{L},\underline{B})$-regularity, which we define as follows:
\begin{definition}\label{def:twistedregularity} 
  We say a sheaf $E$ is \emph{$(r,\underline{L},\underline{B})$-regular} if $E\otimes B_j$ is $r$-regular with respect to $L_j$ for all $j=1,\ldots,j_0$.
\end{definition}

The above idea extends to the case that each $B_j$ is a formal sum of line bundles over $\mathbb R_{\ge 0}$, i.e., where for every $j=1,\ldots,j_0$, we have
\begin{equation}
 B_j = \sum\nolimits_{i=1}^N b_{ji} B_{ji}, \label{eq:formalsum} \end{equation}
where $B_{ji}$ is a line bundle, $b_{ji}\in \mathbb R_{\ge 0}$, and, to avoid trivialities, such that for all $j$ there is some $i$ with $b_{ji} \neq 0$.  
Given such data it is clear there is a (natural) stability parameter which we denote by $\sigma = (\underline{L};B_1,\ldots,B_{j_0})$
with the property that for any sheaf $E$
$$P^{\sigma}_E(k) = \sum\nolimits_j \sum\nolimits_i b_{ji} \chi(E\otimes L_j^k \otimes B_{ji}).$$
Explicitly, one may take $(\underline{L};B_1,\ldots,B_{j_0})$ to be 
$$ \underbrace{\left(L_1,\ldots,L_1\right.}_{N \text{ times }},\ldots,\underbrace{L_{j_0},\ldots,L_{j_0}}_{N \text{ times }},B_{11},\ldots,B_{1N},\ldots,B_{j_01},\ldots,B_{j_0N},  b_{11},\ldots,b_{1N},\ldots,b_{j_01},\ldots,b_{j_0N}).$$
Thus, in terms of the notation of \cite{GRTI} we have $(\underline{L},\sigma_1,\ldots,\sigma_{j_0}) = (\underline{L};\sigma_1 \mathcal O_X,\ldots,\sigma_{j_0}\mathcal O_X)$.

\begin{remark}
  This definition of ``twisted'' stability (for $j_0=1$) is precisely the same as that introduced by Matsuki-Wentworth \cite{MatsukiWentworth}.  However, the point of view is altered slightly: whereas they wish to consider the stability formally as given by the Euler characteristic for a $\mathbb Q$-twist, defined by the Riemann-Roch Theorem, we instead consider it as given by a weighted sum of Euler characteristics of twists by integral bundles. See also \cite[2.3.4]{Lieblich} for an interpretation of this twisted stability using twisted sheaves. 
\end{remark}

\begin{remark}[Riemann-Roch]\label{rem:RRochIII}
We will frequently use the fact that such twisting does not affect the first non-trivial term of the reduced Hilbert polynomial.  In detail, suppose $X$ is smooth of dimension $\dim X=d$, and given $\sigma = (\underline{L},\underline{B},\sigma_{1},\ldots,\sigma_{j_{0}})$ define
$$ \gamma: = \sum\nolimits_{j} \sigma_{j}c_{1}(L_j)^{d-1} \in N_1(X)_\mathbb{\R}$$
(which we observe to be independent of $\underline{B}$).     Then, using the Riemann-Roch theorem we have for any torsion-free sheaf $E$ on $X$ that
$$ p_{E}^{\sigma} = \frac{k^{d}}{d!} + \hat{\mu}^{\sigma}(E) \frac{k^{d-1}}{(d-1)!} + O(k^{d-2}),$$
where 
$$ \hat{\mu}^{\sigma}(E) = C_1 \frac{\int_X c_1(E) \gamma}{\rank E} + C_2$$
and $C_1, C_2$ are the topological constants
\begin{equation*}
 C_1 = \frac{1}{\sum_j \sigma_j \int_X c_1(L_j)^d} \;\text{ and }\;C_2 = \frac{\sum\nolimits_j \sigma_j\int_X( \Todd_1(X) + c_{1}(B_{j})).c_1(L_j)^{d-1}}{\sum_j \sigma_j \int_X c_1(L_j)^d},
\end{equation*}
which are independent of $E$. In particular, precisely as in \Iref{Lemma}{lem:slopeimpliesmulti} for any torsion-free coherent sheaf $E$ the following implications hold:
\begin{center}
slope stable with respect to $\gamma$ $\Rightarrow$ stable with respect to $\sigma$ $\Rightarrow$ semistable with respect to $\sigma$ $\Rightarrow$ slope semistable with respect to $\gamma$.
\end{center}
 Consequently, any boundedness result for a set of sheaves of a given topological type that are slope semistable with respect to $\gamma$ implies boundedness for the set of sheaves of this topological type that are semistable with respect to $\sigma$.
\end{remark}

\section{Uniformity}\label{sect:uniformity}
 Let $X$ be a smooth variety of dimension $d$, fix $\underline{L} = (L_1,\ldots,L_{j_0})$, where each $L_j$ is ample, and let $B_1,\ldots,B_{j_0}$ be line bundles. Moreover, let $\tau \in B(X)_\mathbb{Q}$,  and set
$$\vol(L_j): =\int_X c_1(L_j)^d.$$
We wish to consider the case of a segment of (twisted) stability parameters; that is for $t\in [0,1]$ let
$$\sigma(t) = (\underline{L};\sigma_1(t)B_1,\ldots,\sigma_{j_0}(t)B_{j_0}).$$
\begin{definition}[Stability segment]\label{def:stabilitysegment}
  We say that $(\sigma(t))_{t\in[0,1]}$ is a \emph{stability segment} if each $\sigma_j\colon [0,1]\to \mathbb R_{\ge 0}$ is a linear function, and if  
$$\sum\nolimits_j \vol(L_j) \sigma_j(t) = 1 \quad \text{ for all } t \in [0,1].$$
\end{definition}
\begin{remark}
Observe that since $\sigma_j(\cdot)$ is linear and non-negative, if $t\in (0,1)$ then $\sigma_j(t)>0$ for all $j$ so $\sigma(t)$ is a ``positive'' stability parameter in the language of \cite{GRTI}.    Also observe that in this case for any torsion-free sheaf $E$ the multiplicities are given by
$$r_E^{\sigma(t)}  = \sum\nolimits_j \sigma_j(t) \rank(E) \vol(L_j) = \rank (E),$$ 
and so the multi-Hilbert polynomial and reduced-Hilbert polynomial are related by
$$ p_E^{\sigma}(k) = \frac{P_E^{\sigma}(k)}{\rank (E)}.$$
This is a polynomial in $k$ whose coefficients are linear functions of $t$.  
\end{remark}
It will be convenient later to allow the variation in $t$ to come from the twisting; so suppose $B_j(t)=\sum_i b_{ji}(t) B_{ji}$ for $j=1,\ldots,j_0$ is a formal sum of line bundles $B_{ji}$ whose coefficients $b_{ji}\colon [0,1]\to \mathbb R_{\ge 0}$ depend linearly on $t$.  Then, the condition that
$$ \sigma(t) = (L_1,\ldots, L_{j_0}; B_1(t), \ldots,B_{j_0}(t))$$
is a stability segment becomes
\begin{equation}
 \sum\nolimits_j \rank(B_j(t)) \vol(L_j) = 1 \quad \text{ for all } t \in [0,1],\label{eq:stablityparametertwisted}
\end{equation}
where $\rank(B_j(t)):= \sum_i b_{ji}(t)$.

\begin{definition}[Uniform stability segment]\label{def:uniform}
  We say that a stability segment $(\sigma(t))_{t\in [0,1]}$ is \emph{uniform} if the reduced multi-Hilbert polynomial of any torsion-free sheaf $E$ is of the form
$$ p_E^{\sigma(t)}(k) =  \frac{k^d}{d!} + a_{d-1}(E) k^{d-1} + \cdots + a_1(E)k + a_0(E,t) \quad \text{ for } t\in [0,1]$$
where  $a_{d-1}(E),\ldots,a_{1}(E)$ are independent of $t$ and $a_0(E,t)$ is linear in $t$.

\end{definition}

The importance of the uniformity hypothesis stems from the following simple statement:
\begin{lemma}\label{lem:uniformcompare}
  Let $(\sigma(t))_{[0,1]}$ be a uniform stability segment, and let $E'\subset E$ be torsion-free.  
\begin{enumerate}
\item There is a $k_0$ such that for all $k\ge k_0$ and all $t\in [0,1]$ we have
$$ p_E^{\sigma(t)}\sim p_{E'}^{\sigma(t)} \quad \text{ if and only if }\quad  p_E^{\sigma(t)}(k)\sim p_{E'}^{\sigma(t)}(k), $$
where $\sim$ is any of $\le$ or $<$ or $\ge$ or $>$. 
\item
For fixed $n,m$ the function
$t \mapsto P_{E'}^{\sigma(t)}(n) P_E^{\sigma(t)}(m) - P_{E}^{\sigma(t)}(n) P_{E'}^{\sigma(t)}(m)$
is linear in $t$.
\end{enumerate}
\end{lemma}
\begin{proof}
 Using \eqref{eq:crux}, we conclude the first statement directly from the uniformity hypothesis. For the second one, one observes that the same  hypothesis implies the non-linear terms in $t$ cancel.
\end{proof}

\begin{remark}\label{rmk:closedness}
So, the crux of the above definition is that if $E'\subset E$ is a proper subsheaf then
\begin{equation}
p_E^{\sigma(t)}(k) - p_{E'}^{\sigma(t)}(k) = \tilde{a}_{d-1} k^{d-1} + \cdots + \tilde{a}_1 k + \tilde{a}_0(t),\label{eq:crux}
\end{equation}
where \emph{only} the constant term $\tilde{a}_0(t)$ depends non-trivially on $t$.  This will certainly not hold in general, but below we will give a number of (natural) examples for which it does.

  The property of being uniform gives semicontinuity of the chamber structure on $[0,1]$ defined by the stability parameters $\sigma(t)$.  For suppose that a sheaf $E$ is semistable with respect to $\sigma(t_p)$ for a sequence $t_p\in [0,1]$ that have a limit $\overline{t}$ as $p$ tends to infinity.  Then, without the uniformity assumption there is no reason to expect that $E$ is semistable with respect to $\sigma(\overline{t})$.  To see this, observe that there could, in principle, be a subsheaf $E' \subset E$ such that the difference of reduced multi-Hilbert polynomials is
$$ p_{E'}^{\sigma(t)}(k) - p_E^{\sigma(t)}(k) = (t-\overline{t}) k + 1,$$  
in which case $E'$ would destabilise $E$ for $t=\overline{t}$ but not for any $t<\overline{t}$.  

However, under the assumptions of uniformity and boundedness it will be the case that $E$ is semistable with respect to $\sigma(\overline{t})$.  To see this observe that we clearly need only consider saturated subsheaves $E'\subset E$ with $\hat{\mu}^{\sigma(\overline{t})}(E')\ge \hat{\mu}^{\sigma(\overline{t})}(E)$ and these form a bounded family by Grothendieck's Lemma.  Thus there are only a finite number of topological types of such $E'$, and so we can find a $k_0$ so the conclusion of Lemma \ref{lem:uniformcompare}(1) holds for all such $E'$.  Thus semistability of $E$ with respect to $\sigma(t)$ for $t<\overline{t}$ implies $p_{E'}^{\sigma(t)}(k_0)\le p_E^{\sigma(t)}(k_0)$ for all $t<\overline{t}$, and thus by continuity $p_{E'}^{\sigma(\overline{t})}(k_0)\le p_E^{\sigma(\overline{t})}(k_0)$ as well.  So applying the Lemma again we get $p_{E'}^{\sigma(\overline{t})}\le p_E^{\sigma(\overline{t})}$ and since this holds for all these subsheaves we conclude $E$ is semistable with respect to $\sigma(\overline{t})$ as well.
\end{remark}

We say that a stability segment $(\sigma(t))_{t\in [0,1]}$ is \emph{bounded} 
if the set of sheaves $E$ of topological type $\tau$ that are semistable with respect to $\sigma(t)$ 
for some $t\in [0,1]$ is bounded.  Our goal is to prove the following:

\begin{theorem}(Thaddeus-Flips Through Moduli Space of Sheaves)\label{thm:intermediate}
Let $X$ be smooth, projective and $\tau\in B(X)_{\mathbb Q}$.  Suppose $\sigma(t)_{t\in [0,1]}$ is a bounded and uniform segment of stability parameters and let $t',t''\in (0,1)\cap \mathbb Q$ with $t'<t''$.   Then there exists a finite sequence of rational numbers $t'=t_0<t_1<\ldots<t_N=t''$  so that for $i=0,\ldots,N-1$ there exist Thaddeus-flips
\[\begin{xymatrix}{
 \mathcal{M}_{\sigma(t_i)} \ar[rd] &  & \mathcal{M}_{\sigma(t_{i+1})} \ar[ld] \\
                 &        \mathcal{M}_{\sigma(t'_i)},               & }
\end{xymatrix}
\]
where $t'_i\in (t_i,t_{i+1})$.  Thus, $\mathcal M_{\sigma(t')}$ and $\mathcal M_{\sigma(t'')}$ are related by a finite number of Thaddeus-flips through spaces of moduli spaces of sheaves.
\end{theorem}

\begin{remark}
There is a further generalisation that one can make by letting the $\sigma_j$ themselves be polynomials in $k$ (as considered, for instance, in \cite{Consul}).  As far as the authors can see, this introduces no new difficulties and the existence and variation results we prove carry over.  With this additional flexibility one can of course construct more stability segments, but it does not appear to be  any easier to construct those that are uniform.
\end{remark}

\subsection{Uniform Version of the Le Potier-Simpson Theorem}
One of the main technical results of \cite{GRTI} states if $\sigma$ is a given bounded stability parameter, then for $m\gg n\gg p\gg 0$ a sheaf $E$ (of a given topological type $\tau$) is semistable if and only if the corresponding module $\Hom(T,E)$ is semistable \Iref{Theorem}{thm:mainsemistabilitycomparison}.   We now enhance this result by showing that if $(\sigma(t))_{t \in [0,1]}$ is a bounded and uniform segment of stability conditions then one can choose $m,n,p$ uniformly over all $t\in [0,1]$.

To simplify the discussion,
we will say a sheaf $E$ is \emph{semistable with respect to $t\in [0,1]$} 
if it is semistable with respect to $\sigma(t)$, 
and we let $P^t_E$ and $p^t_E$ denote the corresponding multi-Hilbert polynomial 
and reduced multi-Hilbert polynomial of $E$, respectively. 
By the boundedness hypothesis we may pick $p$ such that any sheaf $E$ of type $\tau$ that is semistable with respect to some $t\in [0,1]$ is $(p,\underline{L},\underline{B})$-regular in the sense of Definition~\ref{def:twistedregularity}.

\begin{theorem}[Uniform version of the Le Potier-Simpson Theorem]\label{thm:lepotiersimpsonuniform}
Let $X$ be a smooth projective variety, and suppose $(\sigma(t))_{t\in [0,1]}$ is a bounded and uniform segment of stability parameters.  Then, if $n\gg p$, for any torsion-free sheaf $E$ of topological type $\tau$ and any $t\in [0,1]$ the following are equivalent:
\begin{enumerate}
\item $E$ is (semi)stable with respect to $t$.
\item $E$ is $(p,\underline{L},\underline{B})$-regular and for all proper $E'\subset E$ we have
  \begin{equation}
 \frac{\sum_j \sigma_j(t) h^0(E'\otimes L_j^n\otimes B_j)}{\rank(E')}(\le) p_E^{t}(n).\label{eq:detectstabilitysubsheafuniform}
 \end{equation}
Moreover, if the saturation of $E'$ in $E$ is not $(n,\underline{L},\underline{B})$-regular, then
  \begin{equation}
 \frac{\sum_j \sigma_j(t) h^0(E'\otimes L_j^n\otimes B_j)}{\rank(E')}\le p_E^{t}(n)-1.\label{eq:detectstabilitysubsheafuniform2}
 \end{equation}
\item $E$ is $(p,\underline{L},\underline{B})$-regular and for all proper saturated $E'\subset E$ with $\hat{\mu}^{\sigma(t)}(E')\ge \hat{\mu}^{\sigma(t)}(E)$ the inequality \eqref{eq:detectstabilitysubsheafuniform} holds.
\end{enumerate}
Moreover, if $E$ is semistable of topological type $\tau$, and $E'\subset E$ is a proper subsheaf, then equality holds in \eqref{eq:detectstabilitysubsheafuniform} if and only if $E'$ is destabilising.  
\end{theorem}
\begin{proof}
This is essentially the same as the proof for a single stability parameter.  Choose $\overline{C}$ and $C_1$, as in the proof of \Iref{Theorem}{thm:detectstabilitysections}, and then choose  $C_2$ so that for any $(p, \underline{L},\underline{B})$-regular sheaf $E$ of type $\tau$ and any $t \in [0,1]$ we have
\begin{equation}
C_2\ge -\hat{\mu}^{\sigma(t)}(E)+1
\end{equation}
and such that additionally
\begin{equation}\label{eq:C21uniform}
  \left(1-\frac{\sum_j \sigma_j(t)}{\rank(E)}\right) (C_1 + \overline{C})+\frac{\sum_j\sigma_j(t)}{\rank(E)}(-C_2 + \overline{C})\le \hat{\mu}^{\sigma(t)}(E)-1
\end{equation}
holds for all $t\in [0,1]$. Then, let $\mathcal S$ be the set of all saturated subsheaves $F\subset E$ where $E$ is $(p,\underline{L},\underline{B})$-regular and $\hat{\mu}^{L_j}(F)\ge -C_2$ for some $1\le j\le j_0$.  This set is bounded, and so for $n\gg p$ we can arrange the conditions (i)-(iv) from that proof to hold uniformly over $t\in [0,1]$.  The condition (i) can be achieved by Lemma \ref{lem:uniformcompare} since there are a finite number of different $p_F$ as $F$ ranges over the bounded set $\mathcal S$, and condition (iv) holds uniformly over $t\in [0,1]$, since \eqref{eq:C21uniform} implies the leading order term in the polynomial on the left hand side of (iv) is less than or equal to the leading term in the right hand side minus 1;  thus, the inequality (iv) holds for all $n$ sufficiently large.

From this point on, the proof is the same as that of \Iref{Theorem}{thm:detectstabilitysections}, 
noting that if $E$ is semistable 
and $E'\subset E$ is a sheaf whose saturation is not $(n,\underline{L},\underline{B})$-regular 
then it is necessarily not of type (B) 
and thus the stronger inequality \eqref{eq:detectstabilitysubsheafuniform2} holds.
\end{proof}

\begin{remark}\label{rmk:LePotierAdvanced}

As is clear from the proof of the (3) implies (1) direction, we actually have that with $n$ as chosen above, if $E$ is pure of dimension $d$ of type $\tau$ and $(p,\underline{L},\underline{B})$-regular, and $E'\subset E$ is saturated and $(n,\underline{L},\underline{B})$-regular with
$$ \frac{\sum_j \sigma_j h^0(E'\otimes L_j^n)}{r_{E'}^{\sigma}}\le p_E^{\sigma}(n)$$
then $p^{\sigma}_{E'}\le p_E^{\sigma}$.
\end{remark}

\subsection{Uniform Comparison of Semistability}\label{sec:uniformcomparison}

We continue as above, so $X$ is smooth and projective of dimension $d$ and we fixed a class $\tau\in B(X)_{\mathbb Q}$.

\begin{theorem}[Uniform comparison of semistability and JH filtrations]\label{thm:semistabilityuniform}
Let $(\sigma(t))_{t\in [0,1]}$ be a bounded and uniform segment of stability parameters and let $t',t''\in (0,1)\cap \mathbb Q$ with $t'<t''$.    Then, for $m\gg n\gg p\gg 0$ and all $t\in [t',t'']$  the following holds  for any sheaf $E$ on $X$ of topological type $\tau$:
 \begin{enumerate}
 \item  $E$ is semistable with respect to $\sigma(t)$ if and only if it is torsion free, $(p,\underline{L},\underline{B})$-regular, and $\Hom(T,E)$ is semistable.
 \item If $E$ is semistable, then
$$ \Hom(T,gr E)\simeq gr \Hom(T,E),$$
where $gr$ denotes the graded object coming from a Jordan-H\"older filtration of $E$ or $\Hom(T,E)$, respectively (taken with respect to $\sigma(t)$).  In particular, two semistable sheaves $E$ and $E'$ are $S$-equivalent if and only if $\Hom(T,E)$ and $\Hom(T,E')$ are $S$-equivalent.  
 \end{enumerate}
\end{theorem}

For the proof we first make explicit our choice of integers.  First, we use boundedness to choose $p$ large enough so that\medskip

\noindent (\hypertarget{C1'}{C1'}) Any sheaf $E$ of topological type $\tau$ that is semistable with respect to some $t\in [0,1]$ is $(p,\underline{L},\underline{B})$-regular.\medskip

Now, let $\mathcal S_2$ be the set of saturated subsheaves $E'\subset E$ where $E$ is $(p,\underline{L},\underline{B})$-regular of topological type $\tau$ and $\hat{\mu}^{\sigma(t)}(E')\ge \hat{\mu}^{\sigma(t)}(E)$ for some $t\in [0,1]$.  Then, $\mathcal S_{2}$ is bounded by Lemma \ref{lem:grothendicklemmavariant} below.  So, we may let $n$ be large enough so\medskip

\noindent (\hypertarget{C2'}{C2'}) (a) The conclusion of the Uniform Version of the Le Potier-Simpson Theorem (Theorem~\ref{thm:lepotiersimpsonuniform}) holds and (b) any sheaf $E'\subset \mathcal S_2$ is $(n,\underline{L},\underline{B})$-regular. \medskip

Now, as in \Iref{Section}{sec:chamber}, the interval $[0,1]$ admits a finite chamber decomposition that witnesses the change of stability as $t$ varies.   In detail, this says there exists a finite sequence $t'=t_0<t_1<\ldots<t_N=t''$ of rational numbers such that any  if $E$ is any $(p,\underline{L},\underline{B})$-regular sheaf and $E'\subset E$ is in $\mathcal S_2$, then $p_{E'}^t\le p_E^t$ for some $t\in (t_i,t_{i+1})$ implies that this holds for all $t\in (t_i,t_{i+1})$.    In particular, if $E$ is semistable with respect to $t\in (t_i,t_{i+1})$, then it is semistable with respect to all $t\in (t_i,t_{i+1})$, and thus by uniformity, Remark \ref{rmk:closedness}, we know that if this holds, then in fact $E$ is semistable with respect to $t_i$ and $t_{i+1}$ as well.

For each $E'\subset E$ where $E$ is $(p,\underline{L},\underline{B})$-regular of topological type $\tau$ and $E'\in \mathcal S_2$, consider the function
$$f_{E'}(t) = P_{E'}^t(n) P_E^t(m) - P_{E}^t(n) P_{E'}^t(m).$$
By boundedness we may enlarge the set $\{t_0,\ldots,t_N\}$ and assume that for all such $E'\subset E$ the following holds:
\begin{equation}
 \text{if }f_{E'}\not\equiv 0, \text{ then }f_{E'}(t) = 0  \text{ for some } t\in[t',t''] \text{ implies } t=t_i \text{ for some }i. \label{eq:extrachambers}
\end{equation}

We now choose $m\gg n$ large enough so that\medskip

\noindent (\hypertarget{C3'}{C3'}) Each $L_j^{-n}$ is $(m,\underline{L},\underline{B})$-regular.\medskip

For the next condition we make a construction completely analogous to that of \Iref{Section}{sect:ComparisonOfSemistability}.  Let $E$ be any sheaf that is $(n,\underline{L},\underline{B})$-regular and has topological type $\tau$. For each $j$ let
$$ \epsilon_j \colon H^0(E\otimes L_j^n\otimes B_{j})\otimes L_j^{-n}\otimes B_{j}^{-1} \to E$$
be the natural (surjective) evaluation maps.  
\begin{definition}\label{def:Esum}
  For an $(n,\underline{L},\underline{B})$-regular sheaf $E$ and subspaces $V_j'\subset H^0(E\otimes L_j^n\otimes B_{j})$ let $E'_j$ and $F'_j$ be the image and kernel of $\epsilon_j$ restricted to $V'_j$, so there is a short exact sequence $ 0\to F_j' \to V'_j\otimes L_j^{-n} \to E_j'\to 0$. Then, define a subsheaf of $E$ by $$ E_{\text{sum}} := E_{\text{sum}} (V_1',\ldots,V'_{j_0}) := E'_{1} + \cdots + E'_{j_0}$$
and let $K=K(V_1,\ldots,V_{j_0})$ be the kernel of the surjection $\bigoplus_j E'_j \to E_{\text{sum}}.$  We let $\mathcal S_1$ be the bounded set of all sheaves $E_j',F'_j,E_{sum}$ and $K$ that arise in this way.
\end{definition}

By increasing $m$ if necessary, we may assume the following.\medskip

\noindent (\hypertarget{C4'}{C4'}) All the sheaves in set $\mathcal S_1$ are  $(m,\underline{L},\underline{B})$-regular.  \medskip

Note the assumption (\hyperlink{C2'}{C2'})(a) implies the usual (non-uniform) Le Potier-Simpson Theorem holds \Iref{Theorem}{thm:detectstabilitysections}.   Since by (\hyperlink{C2'}{C2'})(b) any sheaf in $\mathcal S_2$ is certainly $(m,\underline{L},\underline{B})$-regular, we see that the conditions so far imply the (twisted versions of) conditions (\hyperlink{C1}{C1})-(\hyperlink{C4}{C4}) from \Iref{Section}{sect:ComparisonOfSemistability} hold for each $t\in [t',t'']$. 

The analogue of condition (\hyperlink{C5}{C5}) is more subtle, for this we require  the following:\medskip

	\noindent (\hypertarget{C5'}{C5'}) For each $i$, let $t_i'$ be a point in the interval $(t_i,t_{i+1})$.  Then, for any $t\in \{t_i,t_i'\}$ the twisted version of condition (C5) from \Iref{Section}{sect:ComparisonOfSemistability} holds.  That is,  if $P_j(k) = \chi(E\otimes L_j^k\otimes B_j)$, then for any integers $c_j\in \{ 0,\ldots,P_j(n)\}$ and sheaves $E'\in \mathcal S_1\cup \mathcal S_2$ the polynomial relation $ P_E^{t}(\sum_j \sigma_j^t)c_j \sim P_{E'}^{t} P_E^{t}(n)$ is equivalent to the relation $ P_E^{t} (m) \sum_j \sigma_j(t) c_j \sim P_{E'}^{t}(m) P_{E}^{t}(n)$, where $\sim$ is any of $\le$ or $<$ or $=$.\medskip
\medskip

We note that (\hyperlink{C5'}{C5'}) is possible for the same reason that (\hyperlink{C5}{C5}) was possible, since we are only demanding that it hold for a \emph{finite} number of stability parameters.   Thus, we have that (the twisted versions of) conditions (C1) though (C5) holds for all the $\sigma(t_i)$ and $\sigma(t_i')$.  In particular, the proof given in \cite[Section 8]{GRTI} applies to establish the Comparison of Semistability and of JH filtrations, \cite[Theorem 8.1]{GRTI}, at these finite collection of points, giving the following:

\begin{theorem}
\label{thm:absolutecase}
The conclusion of Theorem \ref{thm:semistabilityuniform} holds if $t$ is contained in the finite set $\{t_{i},t_{i}'\}$.
\end{theorem}

We next claim that by enlarging  $m$ if necessary, we may also assume the following holds for all $t\in [0,1]$:\medskip

\noindent (\hypertarget{C6'}{C6'}) If $t\in [0,1]$
and $E'\subset E$ with $E$ $(p,\underline{L},\underline{B})$-regular of topological type $\tau$ and $E'\in \mathcal S_1$ is such that
\begin{equation}
\sum\nolimits_j \sigma_j(t) h^0(E'\otimes L_j^n\otimes B_j) \rank(E) \le P^{t}_E(n)\rank(E') - \frac{1}{\max_j \vol(L_j)}, \label{eq:c61}
\end{equation}
then 
\begin{equation}
 \sum\nolimits_j \sigma_j(t) h^0(E'\otimes L_j^n\otimes B_j) P_E^{t}(m) \le P^{t}_E(n) P_{E'}^{t}(m) -1. \label{eq:c62}
\end{equation}

To see this, observe that as $\mathcal S_1$ is bounded, the set of multi-Hilbert polynomials $P_{E'}^{\sigma(t)}$ that arise for $E' \in \mathcal{S}_{1}$ is finite.  Now, \eqref{eq:c61} says that the corresponding inequality in the leading order coefficient in $m$ (namely the coefficient of $m^{d}$) in \eqref{eq:c62} holds strictly (and by an amount bounded away from $0$ for all $t\in [0,1]$).  So, as the coefficients of all the lower order terms are bounded over $t\in [0,1]$ for all such $E'$, we conclude that \eqref{eq:c62} holds for all $m\gg 0$.\medskip

Having made our choice of integers $m,n,p$ we turn to some preliminary lemmas needed for our proof of the Uniform Comparison of Semistability.   For a non-trivial $A$-module $M = \bigoplus_j V_j\oplus W_j$ we set
$$ \mu_t(M) = \frac{\sum_j \sigma_j(t) \dim V_j}{\sum_j \sigma_j(t) \dim W_j}\quad \text{ for  }t\in (0,1) .$$ Moreover, we recall the following technical definition from \cite{GRTI}.
\begin{definition}
  Let $M'=\bigoplus_j V_j'\oplus W_j'$ and  $M''=\bigoplus_j V_j''\oplus W_j''$ be two submodules of a given $A$-module $M$.  We say that $M'$ is \emph{subordinate} to $M''$ if
  \begin{equation}\label{eq:subordinate}
 V_j'\subset V_j'' \text{ and } W''_j\subset W'_j \quad \text{ for all } j.
\end{equation}
\end{definition}

\begin{lemma}\label{lem:crucial}
Let $E$ be a torsion-free sheaf of type $\tau$.  Suppose $t\in (0,1)$ is such that $E$ is semistable with respect $\sigma(t)$.   Then, if $M'$ is a submodule of $M=\Hom(T,E)$ with $\mu_t(M')= \mu_t(M)$, there exists an $(n,\underline{L},\underline{B})$-regular sheaf $E'\subset E$ such that $M'$ is subordinate to $\Hom(T,E')$.
\end{lemma}

\begin{proof}
Write $M' = \oplus_j V_j'\oplus W_j'$ and $E' = E_{sum}(V_1',\ldots,V_{j_0}')$.  From \Iref{Proposition}{prop:tightmodulescomefromsheaves} we know that $M'$ is subordinate to $\Hom(T,E')$.  So, the issue is to prove that $E'$ is $(n,\underline{L},\underline{B})$-regular. To this end, we first claim that
\begin{equation}
 \sum\nolimits_j \sigma_j(t) h^0(E'\otimes L_j^n\otimes B_j) \rank(E) > P_E^t(n)\rank(E') -\frac{1}{\max_j \vol(L_j)}.\label{lem:crucialineq1}
\end{equation}
To see this, suppose for contradiction that in fact
$$ \sum\nolimits_j \sigma_j(t) h^0(E'\otimes L_j^n\otimes B_j) \rank(E) \le P_E^t(n)\rank(E') -\frac{1}{\max_j \vol(L_j)}.$$
By (\hyperlink{C4'}{C4'}) we know $E'$ is $(m,\underline{L},\underline{B})$-regular, and so using  condition (\hyperlink{C6'}{C6'}) we have
\begin{align*}
  \mu_t(\Hom(T,E')) = \frac{\sum_j \sigma_j(t) h^0(E'\otimes L_j^n\otimes B_j)}{P_{E'}^t(m)} &\le   \frac{P_E^{t}(n)}{P_E^t(m)}-\frac{1}{P_E^t(m)P_{E'}^t(m)}\\
  &< \frac{P_E^{t}(n)}{P_E^t(m)} =
  \mu_t(M).
\end{align*}
 On the other hand $M'$ is subordinate to $\Hom(T,E')$, so certainly $\mu_t(\Hom(T,E'))\ge \mu_t(M')=\mu_t(M)$ which gives the desired contradiction.

Thus, by the assumption that $E$ is semistable and by part a) of (\hyperlink{C2'}{C2'}) the ``(1) implies (2)'' direction of the Uniform Le Potier-Simpson Theorem (Theorem \eqref{thm:lepotiersimpsonuniform}) holds, so we know that the saturation $F$ of $E'$ is $(n,\underline{L},\underline{B})$-regular.  So, it is sufficient to prove that $E'$ is saturated.

To this end, observe that if $h^0(E'\otimes L_j^n\otimes B_j) = h^0(F\otimes L_j^n\otimes B_j)$ for some $j$, 
then $F= E'$ as $F\otimes B_j\otimes L_j^n$ is globally generated, and so we are done.  So, suppose this is not the case, i.e., $h^0(E'\otimes L_j^n\otimes B_j)\le h^0(F\otimes L_j^n\otimes B_j)-1$ for all $j$.  Then, using $\rank(E') = \rank(F)$ we obtain
$$\frac{\sum_j \sigma_j(t) h^0(E'\otimes L_j^n\otimes B_j)}{\rank(E')}  \le\frac{\sum \sigma_j(t) h^0(F\otimes L_j^n\otimes B_j)}{\rank(F)} - \frac{\sum_j \sigma_j(t)}{\rank(F)}.$$  As $E$ is semistable, we can again apply the ``(1) implies (2)'' direction of the Le Potier-Simpson (Theorem \ref{thm:lepotiersimpsonuniform})  to give that the right hand side of the previous equation is in turn bounded by
$$ \frac{P_E^t(n)}{\rank(E)} - \frac{\sum_j \sigma_j(t)}{\rank(F)} = \frac{P_E^t(n)}{\rank(E)} - \frac{\sum_j \sigma_j(t)}{\rank(E')},$$
and so
\begin{align*}
\sum\nolimits_j \sigma_j(t) h^0(E'\otimes L_j^n\otimes B_j) \rank(E) &\le P_E^t(n) \rank(E') - \sum\nolimits_j \sigma_j(t) \rank(E)\\
&\le P_E^t(n) \rank(E') - \frac{1}{\max_j \vol(L_j)},
\end{align*}
as $\sum_j \sigma_j(t) \cdot \max_j \vol(L_j) \ge \sum_j \sigma_j(t) \vol(L_j) =1$.  This contradicts the above claim and therefore completes the proof.
\end{proof}

The next lemma is a slight refinement of our Comparison of Semistability Theorem for a single stability parameter, which could just as well have been proved in \Iref{Section}{sect:ComparisonOfSemistability}.

\begin{lemma}\label{lem:comparisonimproved}
Let $t=t_i$ for some $i$.  Suppose that $E$ is $(p,\underline{L},\underline{B})$-regular of topological type $\tau$ and that $E'\subset E$ has $p_{E'}^{t} > p_E^{t}$.   Then, $\mu_{t}(\Hom(T,E')) >\mu_{t}(\Hom(T,E))$.  
\end{lemma}
\begin{proof}
We have $\hat{\mu}_{E'}^{t_i} \ge \hat{\mu}_{E}^{t_i}$ and so $E'\in \mathcal S_{2}$, which implies $E'$ is $(n,\underline{L},\underline{B})$-regular by (\hyperlink{C2'}{C2'}).  Suppose for contradiction $\mu_{t}(\Hom(T,E'))\le \mu_t(\Hom(T,E))$.  Then,
$$ \sum\nolimits_j \sigma_j(t) h^0(E'\otimes L_j^n\otimes B_j) P_E^t(m) \le \sum\nolimits_j \sigma_j(t) h^0(E'\otimes L_j^m\otimes B_j) P_E^t(n).$$
Now by (\hyperlink{C5'}{C5'}) (since $t$ is assumed to be among the $\{t_i\}$) we deduce
$$P_E^t \sum\nolimits_j \sigma_j(t) h^0(E'\otimes L_j^n\otimes B_j)  \le P_{E'}^t P_E^t(n)$$
and so taking the leading order term in these polynomials
$$ r_E^t \sum\nolimits_j \sigma_j(t) h^0(E'\otimes L_j^n\otimes B_j) \le r_{E'}^t P_E^t(n).$$
Thus, by the Le Potier-Simpson Theorem (as discussed in Remark \ref{rmk:LePotierAdvanced}), this implies $p_{E'}^t\le p_E^t$, which contradicts our choice of $E'$.
\end{proof}

\begin{lemma}\label{lemma:linalg}
Suppose an $A$-module $M$ is semistable with respect to $t'\in (0,1)$ and not semistable with respect to $t''\in (0,1)$.  Then there exists a $t$ between $t'$ and $t''$ and a submodule $M'$ of $M$ such that 
\begin{enumerate}
\item $M$ is properly semistable with respect to $t$.  
\item $\mu_t(M') = \mu_t(M)$.
\item The function $s\mapsto \mu_s(M') - \mu_s(M)$ is not identically zero.
\end{enumerate}

\end{lemma}
\begin{proof}
Assume $t'<t''$ (the other case being proved in the same way).  The set $\mathcal D$ of dimension vectors $\underline{e}$ of non-zero submodules $M'$ of $M$ is finite.  For each such $\underline{e}\in \mathcal D$ set 
$$g_{\underline{e}}(s) := \mu_s(M') - \mu_s(M) \text{ for } s\in [t',t'']$$ 
and
$$ \mathcal D_0 := \{ \underline{e} \in \mathcal D : g_{\underline{e}}(\cdot) \text{ is not identically } 0\}$$
As $M$ is not semistable with respect to $t''$ the set $\mathcal D_0$ is non-empty.

Now if $\underline{e}\in \mathcal D_0$ we have $g_{\underline{e}}(t')\le 0$ as $M$ is semistable with respect to $t'$.   Define
$$ t:=\sup \{ s\in [t',t'') : g_{\underline{e}}\le 0 \text{ on } [t',s] \text{ for all } \underline{e}\in \mathcal D_0\}.$$
Then clearly $M$ is semistable with respect to $t$, and there must be some $\underline{e}\in \mathcal D_0$ for which $g_{\underline{e}}(t)=0$.    Letting $M'$ be a submodule with dimension vector $\underline{e}'$ proves the lemma.
\end{proof}

\begin{proof}[Proof of Theorem \ref{thm:semistabilityuniform}]

Suppose first $E$ is semistable with respect to some $t\in [t',t'']$.  Then by definition it is torsion-free, and it is $(p,\underline{L},\underline{B})$-regular by choice of $p$.   We aim to show that $M:=\Hom(T,E)$ is also semistable with respect to $t$.  

To this end, suppose first that $t=t_{i}$ for some $i$.  Then, certainly $M$ is semistable by the Preservation of Semistability with respect to these points (cf.\ Theorem \ref{thm:absolutecase}).  Thus, we may assume $t\in (t_{i},t_{i+1})$ for some $i$.  Then, $E$ must by also semistable with respect to all $t'\in [t_i,t_{i+1}]$, since by construction $t_i$ were the walls at which semistability of sheaves of this type may change and our stability segment is assumed to be uniform.  In particular this is true at the endpoints, so once again by Theorem \ref{thm:absolutecase} the module $M$ is semistable with respect to $t_i$ and $t_{i+1}$.

Suppose for contradiction $M$ is not semistable with respect to this $t$.  By Lemma \ref{lemma:linalg} there is some $t'\in [t_i,t_{i+1}]$ and a submodule $M'\subset M$ such that $M$ is properly semistable with respect to $t'$, $\mu_{t'}(M') = \mu_{t'}(M)$, and so $s\mapsto \mu_s(M')-\mu_s(M)$ is not identically zero. Note that this function is not linear in general.

 Thus, we can apply Lemma \ref{lem:crucial} to deduce there is an $(n,\underline{L},\underline{B})$-regular subsheaf $E'\subset E$ such that $M'$ is subordinate to $\Hom(T,E')$.  So
\begin{equation}
\mu_{t'}(\Hom(T,E')) \ge \mu_{t'}(M') =\mu_{t'}(M)\label{eq:linearfunction1}
\end{equation}
and hence $\mu_{t'}(\Hom(T,E')) = \mu_{t'}(M)$ by semistability of $M$ with respect to $t'$.  So $\mu_{t'}(M') = \mu_{t'}(\Hom(T,E'))$ and thus in fact $M'= \Hom(T,E')$.

Now by Lemma~\ref{lem:uniformcompare}, since we are assuming that our stability segment is uniform,  the function 
\begin{equation} \label{eq:crossmult}
 f_{E'}(s)=  P_{E'}^s(n) P_E^s(m) - P_{E}^s(n) P_{E'}^s(m)
\end{equation}
is linear in $s$, and by  $(n,\underline{L},\underline{B})$-regularity of both $E'$ and $E$ 
\begin{align}\label{eq:fE}
f_{E'}(s)&=  P_{E'}^s(m) P_{E}^s(m) \left( \frac{P_{E'}^s(n)}{P_{E'}^s(m)} -  \frac{P_{E}^s(n)}{P_{E}^s(m)}\right) \\&=   P_{E'}^s(m) P_{E}^s(m) \left( \mu_s(\Hom(T,E')) - \mu_s(M)\right).\label{eq:fE2}
\end{align}
But this is absurd, since by the above $f_{E'}(\cdot)$ is linear and not identically zero, $f_{E'}(t')= 0$, but $f_{E'}(t_i)\le 0$ and $f_{E'}(t_{i+1})\le 0$, as $\Hom(T,E)$ is semistable at $t_i$ and $t_{i+1}$.   Thus, we have shown that $M$ is semistable with respect to $t$, as required.

For the converse, suppose $M$ is semistable with respect to some $t\in [t_i,t_{i+1}]$ and we wish to show that $E$ is semistable with respect to $t$.  

As above, by Theorem \ref{thm:absolutecase} we are done if  $t$ is in the set $\{t_i\}$, so we may assume $t\in (t_i,t_{i+1})$.  We suppose for contradiction $E$ is not semistable with respect to $t$.    Then, there exists a saturated $E'\subset E$ with $p_{E'}^t>p_E^t$.    So certainly $\hat{\mu}^t_{E'}\ge \hat{\mu}^t_E$ and by (\hyperlink{C2'}{C2'}) $E'$ is then $(n,\underline{L},\underline{B})$-regular.  
Now clearly we must either have $p_{E'}^{t_{i}}>p_E^{t_i}$ or  $p_{E'}^{t_{i+1}}>p_E^{t_i+1}$ by linearity of the $\sigma_j$ and by uniformity of $(\sigma(t))_{t\in [0,1]}$.
From Lemma \ref{lem:comparisonimproved} this implies either $\mu_{t_i}(\Hom(T,E'))>\mu_{t_i}(M)$ or $\mu_{t_{i+1}}(\Hom(T,E'))>\mu_{t_{i+1}}(M)$.  Thus, with $f_{E'}$ as in \eqref{eq:crossmult} above we deduce from \eqref{eq:fE2} that $f_{E'}(t_i)>0$ or $f_{E'}(t_{i+1})>0$.   In particular, we conclude $f_{E'}$ is not identically zero.    But $M$ was assumed to be semistable with respect to $t$, so certainly $f_{E'}(t)\le 0$, which by our choice of $\{t_i\}$ implies $t=t_i$ or $t=t_{i+1}$ which is absurd.  Thus, $E$ is semistable as claimed.

Finally, we turn to the statement about $S$-equivalence.  By Theorem \ref{thm:absolutecase}, the statement we want holds at any point in $\{t_i\}$, so our interest lies in the open interval $(t_i,t_{i+1})$.   Again by construction, the $S$-equivalence class of any semistable $E$ taken with respect to any point within $(t_i,t_{i+1})$ is independent of that point.  We claim that the same is true for the semistable module $M=\Hom(T,E)$.

To see this, suppose that $M$ is semistable with respect to some point $t\in (t_i,t_{i+1})$ and that $M'\subset M$ is destabilising with respect to $t$.  

By Lemma \ref{lem:crucial} we may find a subsheaf $E'$ of $E$ such that $M'$ is subordinate to $\Hom(T,E')$ and thus $\mu_t(M')\le\mu_t(\Hom(T,E'))$ \Iref{Lemma}{lem:tightslope}.  Hence, we may suppose that $M'= \Hom(T,E')$ for some $E'\subset E$ that is $(n,\underline{L},\underline{B})$-regular.  So the function $f_{E'}(\cdot)$ is linear.  Thus, we have $f_{E'}(t)=0$ but $f_{E'}(\cdot)\le 0$ over all of $(t_i,t_{i+1})$ by semistability of $M$.  Therefore, $f_{E'}\equiv 0$, and so $M'$ is a destabilising submodule of $M$  with respect to any point in $(t_i,t_{i+1})$.

Now, by Theorem \ref{thm:absolutecase} applied to the chosen points $t_i'\in (t_i,t_{i+1})$ we have $\Hom(T,gr(E))$ is isomorphic to $gr(\Hom(T,E))$.  Hence, this must also hold for all points in $(t_i,t_{i+1})$, as the isomorphism class of both sides are unchanged within this interval, completing the proof.
\end{proof}

\begin{remark}\label{remark:chambers}
  The above argument has an interpretation in terms of the various chamber structures that are in play.  Let $\Sigma\subset (\mathbb R_{\ge 0})^{j_0}\setminus\{0\}$ be a bounded subset of stability parameters.  Then, as we have seen, $\Sigma$ has a linear chamber structure that witnesses the change in multi-Gieseker-stability as $\sigma\in\Sigma$ varies.  
On the other hand, representations $M$ of the corresponding quiver are parametrised by a representation space $R$ on which a product $G$ of general linear groups act, and here GIT-stability has a linear chamber structure on $\Pi:=\mathbb Q^{2j_0}\setminus\{0\}$ obtained as the character $\chi_{\theta}$ varies with $\theta\in \Pi$.    

We recall the assignment $\sigma \mapsto \chi_{\theta}$ defined in \Iref{Section}{subsubsect:stabilityandmodulispaces} was given by
$$ \theta_{v_j} = \frac{ \sigma_j}{\sum_i \sigma_i d_{v_i}}\quad \theta_{w_j} = \frac{ -\sigma_j}{\sum_i \sigma_i d_{w_i}},$$
which is in general \emph{non-linear} in $\sigma$.  Thus, in general, the linear walls in the space $\Pi$ will not pullback to the linear walls in $\Sigma$, and then there is no reason to expect the variation of GIT to properly reflect the variation of definition of multi-Gieseker-stability.

However, what is essentially proved above, is that if $(\sigma(t))_{t \in [0,1]}$ is uniform, then the (non-trivial) walls in $\Pi$ do in fact pullback to linear walls in $\Sigma$.  
This is because (Lemma \ref{lem:crucial}) the non-trivial walls (namely those for which there is an $A$-module of the form $\Hom(T,E)$ that is properly semistable) are detected by $(n,\underline{L},\underline{B})$-regular subsheaves $E'\subset E$, and from Lemma \ref{lem:uniformcompare}, the pullback of this wall is actually linear in $\Sigma$.     
In fact, it is possible to argue more generally and instead of using a segment of stability parameters prove that, under a suitable uniformity hypothesis again, that the two chambers structures essentially agree for $m\gg n\gg p\gg 0$.
\end{remark}

\subsection{Uniform Variation (Proof of Theorem \ref{thm:intermediate})}

The proof of Theorem \ref{thm:intermediate}, which identifies the intermediate spaces, is essentially the same as that for a finite set of stability parameters as in \Iref{Section}{subsect:MultiGiesekerVariation}.   There, we considered the union $Y = \bigcup_{\sigma\in \Sigma'} Q^{[\sigma\text{-}ss]}\subset R$ as $\sigma$ varies in a finite set $\Sigma'$ of stability parameters, together with its closure $Z = \overline{Y}$.  This is now replaced by $Y = \bigcup_{t\in  [t',t'']} Q^{[\sigma(t)\text{-}ss]}$, which is still a finite union as $\sigma(t)_{t\in [t',t'']}$ admits a finite chamber structure, and its closure $Z = \overline{Y}$.  By the Uniform Comparison of Semistability, Theorem~\ref{thm:semistabilityuniform},  the ``master space'' statement \Iref{Theorem}{thm:masterspace} $$Z^{\sigma(t)}:= R^{\sigma(t)\text{-}ss}\cap Z = Q^{[\sigma(t)\text{-}ss]}$$ now holds for all $t\in [t',t'']$.  In \Iref{Corollary}{cor:masterspace} we identified a cone $\mathcal C_G(Z)$ together with a chamber decomposition reflecting the change of sets of semistable points.  So, as $\sigma(t)$ varies with $t$, we have a path $\theta_{\sigma(t)}$ of characters which pass through the various chambers, that precisely witness the Thaddeus-flips that occur.  By the uniform statements in the previous section we know that each of these is a moduli space of sheaves, precisely as in proof of \Iref{Corollary}{cor:masterspace}.\hfill \qed

\part{Chamber Structures}

\section{Chamber structure for Gieseker-stability}\label{sec:chambrstructuregieseker}

Our goal here is to exhibit a chamber structure on the ample cone of a projective manifold that witnesses the change in Gieseker-stability as the polarisation varies.  For this, assume that $X$ is smooth of dimension $d$ over an algebraically closed field of characteristic zero, and let $\tau\in B(X)_{\mathbb Q}$.  We recall the slope of a torsion-free sheaf $E$ on $X$, as in \Iref{Definition}{def:slopstability}, with respect to $L\in \Amp(X)_{\mathbb R}$, or a curve class $\gamma\in N_1(X)_{\mathbb R}$ is
$$ \mu^L(E) = \frac{\int_X c_1(E) c_1(L)^{d-1}}{\rank(E)} \text{ and }  \mu_{\gamma}(E)=\frac{\int_X c_1(E)\cdot \gamma}{\rank(E)}.$$
We also recall $\Pos(X)_{\R}:= \{ \gamma\in N_{1}(X)_{\R}: \gamma = D^{d-1}\text{ for some }D\in \Amp(X)_{\R}\}$ and that the map $p\colon \Amp(X)_{\mathbb R}\to \Pos(X)_{\mathbb R}$ given by $p(x) = x^{d-1}$ is a homeomorphism \cite[Prop.~6.5]{GrebToma}. Observe that by definition for any torsion-free sheaf $E$ and any $L\in \Amp(X)_{\mathbb R}$ we have $\mu^L(E) = \mu_{p(L)}(E)$ . 

Now, suppose that $K\subset \Amp(X)_{\R}$ is open and relatively compact and let
$$ \Cone(K): = \{ \lambda L\in \Amp(X)_{\R} \mid L\in K\text{ and } \lambda\in \R_{>0}\}$$
be the cone over $K$ with the origin removed.

\begin{lemma}\label{lem:grothendicklemmavariant}
Let $\mathcal S'$ be a bounded family of torsion-free sheaves of topological type $\tau$.  Then, the set
\begin{equation*}
  \mathcal S := \left\{
F \left|\begin{split} & F \text{ is a saturated subsheaf of some } E\in \mathcal S' \text{ and } \\
&\mu_{\gamma}(F) \ge \mu_{\gamma}(E) \text{ for some } \gamma \text{ in the convex hull of } p(\Cone(K))
\end{split}\right.\right\}
\end{equation*}
is bounded.
\end{lemma}
\begin{proof}
Observe the inequality $\mu_{\gamma}(F)\ge \mu_{\gamma}(E)$ is invariant under replacing $\gamma$ by a positive multiple. So, as $\operatorname{Convexhull}(p(\Cone(K))\subset \Cone(\operatorname{Convexhull}(p(K))$, it is sufficient to prove the theorem with $\Cone(K)$ replaced with $K$.  

Fix $\tilde{\gamma}$ in the convex hull of $p(\overline{K})$.    Then, we can find $\gamma_1,\ldots,\gamma_m\in \Pos(X)_{\mathbb R}$ so that $\gamma_i = p(A_i)$ for some $A_i\in \Amp(X)_{\mathbb Q}$ and so $\tilde{\gamma}$ lies in the interior of the convex hull $H$ of $\gamma_1,\ldots,\gamma_m$.    We claim that the set 
$$\mathcal T_H:=\{ F : F \text{ is a saturated subsheaf of an } E\in \mathcal S' \text{ and } \mu_{\gamma}(F) \ge \mu_{\gamma}(E) \text{ for some } \gamma\in H\}$$
is bounded.  To prove this, suppose $\gamma$ is a convex combination of $\gamma_1,\ldots,\gamma_m$.  If $\mu_{\gamma}(F) \ge \mu_{\gamma}(E)$, then by convexity $\mu_{\gamma_i}(F) \ge \mu_{\gamma_i}(E)$ for some $i$, which says precisely $\mu^{A_i}(F) \ge\mu^{A_i}(E)$.  Thus, $\mathcal T_H$ is contained in a finite union of sets that are each bounded by Grothendieck's Lemma, \cite[Lem.~1.7.9]{Bible} or \cite[Th\'eor\`eme 2.2]{Groth}, proving the claim.

Now, the convex hull of the compact set $p(\overline{K})$ is also compact (a simple corollary of Carath\'eodory's theorem) and so can be covered with a finite union of sets of form $H$. Hence, we deduce that $\mathcal S$ is contained in a finite union of sets of the form $\mathcal T_H$, and thus is also bounded as required.
\end{proof}

We apply the preceding discussion as follows.  Define
\begin{equation}\label{eq:mathcalS'K}
  \mathcal S'_K := \left\{E \left|\begin{split}  & E \text{ is torsion-free of topological type } \tau \text{ that is }\\
& \text{slope semistable with respect to some } \gamma\in p(\Cone(K))
\end{split} \right.\right\}.
\end{equation}
Then, as slope semistability is invariant under scaling $\gamma$ by a positive multiple, this set is unchanged if $\Cone(K)$ is replaced with $K$.  Thus, $\mathcal S'_{K}$ is bounded by \Iref{Theorem}{thm:K^+}.  Consequently, by Lemma \ref{lem:grothendicklemmavariant} the set
\begin{equation}\label{eq:mathcalSK}
  \mathcal S_K := \left\{
F \left|\begin{split} & F \text{ is a saturated subsheaf of some } E\in \mathcal S'_K \text{ and } \\
&\mu^L(F) \ge \mu^L(E) \text{ for some } L\in \Cone(K)
\end{split}\right.\right\}
\end{equation}
is also bounded.  Now,  for each $F\in \mathcal S_{K}$ write the corresponding difference of reduced Hilbert-polynomials as
\begin{equation}
p_{F}^L(k) - p_E^L(k)  =\frac{1}{\vol(L)}\sum_{i=1}^{d} \beta_{F,i}^L \frac{k^{d-i}}{(d-i)!},\label{eq:definitionbeta}
\end{equation}
where $E$ is any sheaf of topological type $\tau$ and $\vol(L):=\int_{X}c_1(L)^{d}$.  Note that here we allow $L$ to be real, so these reduced Hilbert polynomials should be defined using the Riemann-Roch theorem (as discussed in \cite[Definition 11.1]{GRTI}, i.e.,
$$p_{F}^{L}(k) := \frac{1}{\rank(F)\vol(L)} \int_{X} e^{kc_1(L)}\ch(F)\Todd(X).$$
Hence, $L\mapsto \beta_{F,i}^L$ is a polynomial function on  $\Amp(X)_{\mathbb R}$ of degree $d-i$ (so in particular $\beta_{F,d}^{L}$ is independent of $L$). 

Moreover, for $i=1,\ldots,d-1$ and $F\in \mathcal S_K$ we set
$$ \tilde{W}_{F,i} = \{ L\in \Amp(X)_{\mathbb R} \mid \beta_{F,i}^L = 0\}$$
which we shall refer to as a \emph{wall}.  So, each such wall is either empty, all of $\Amp(X)_{\mathbb R}$, or a non-trivial real algebraic variety in $\Amp(X)_{\mathbb R}$.

\begin{definition}\label{def:mathcaltildeW}
We let $\tilde{\mathcal W}_K$ be the set of walls $\tilde{W}_{F,i}$ for $F\in \mathcal S_K$ and $i=1,\ldots,d-1$ such that $\tilde{W}_{F,i}$ is neither empty nor all of $\Amp(X)_{\mathbb R}$.
\end{definition}

Observe that $\tilde{W}_{F,i}$ depends only on the topological type of $F$, and by boundedness of $\mathcal S_{K}$ there are only finitely many such types.  Hence, the set $\tilde{\mathcal W}_{K}$ consists of   a finite number of (non-trivial) walls that divide $\Cone(K)$ into a number of chambers (see \cite[Definition 4.1]{GRTI} for the precise definition of a chamber).

\begin{proposition}[Existence of chamber structure for Gieseker-stability]\label{prop:chamberstructuregieseker}
The collection $\tilde{\mathcal W}_K$ of walls gives a chamber structure on $\Cone(K)$ that witnesses the change in Gieseker-stability as $L$ varies within $\Cone(K)$.  More precisely, if $L',L''\in \Cone(K)$ lie in the same chamber, then for any sheaf $E$ of topological type $\tau$ it holds that
\begin{enumerate}
\item $E$ is Gieseker-(semi)stable with respect to $L'$ if and only if it is Gieseker-(semi)stable with respect to $L''$.
\item If $E$ is Gieseker-(semi)stable with respect to $L'$ and $L''$ then the $S$-equivalence class of $E$ is the same taken with respect to either $L'$ or $L''$.
\end{enumerate}
\end{proposition}
\begin{proof}
The proof of this is essentially the same as that of \Iref{Proposition}{prop:weakchamberstructure}.  Let $E$ be of topological type $\tau$ and suppose for contradiction $E$ is Gieseker-semistable with respect to $L'$ and not Gieseker-semistable with respect to $L''$.  

Then, on the one hand $E$ is slope semistable with respect to $p(L')\in p(\Cone(K))$, and on the other hand there exists a saturated $F\subset E$ such that $p_F^{L''}>p_E^{L''}$.  In particular, this implies $\mu^{L''}(F)\ge \mu^{L''}(E)$ (cf.~ \Iref{Example}{ex:RRochI}), and so by definition $F\in \mathcal S_K$.  As above, write the difference of reduced Hilbert-polynomials as 
$$p_{F}^L(k) - p_E^L(k)  =\frac{1}{\vol(L)}\sum_{i=1}^{d} \beta_{F,i}^L \frac{k^{d-i}}{(d-i)!}.$$
So, we have $p_{F}^{L''}>p_E^{L''}$, but by stability of $E$ with respect to $L'$ also  $p_{F}^{L'}\le p_E^{L'}$.  Let $i$ be the largest integer such that $\beta_{F,j}^{L'}=\beta_{F,j}^{L''}=0$ for all $j<i$.  Then, $\beta_{F,i}^{L'}\le 0$ and $\beta_{F,i}^{L''}\ge 0$ (but not both equal to zero by choice of the index $i$).  This implies three things.  First, by continuity of the function $L\mapsto \beta_{F,i}^L$, and by connectedness of a chamber, some point of this chamber must be contained in $\tilde{W}_{F,i} = \{L \mid \beta^L_{F,i}=0\}$.  Second, the wall $\tilde{W}_{F,i}$ is non-trivial, so lies in $\tilde{\mathcal W}_K$.  And third, there exists a point in the chamber not contained in $\tilde{W}_{F,i}$.  Together, these three statements contradict the definition of what it means to be a chamber. 

The statement for Gieseker-stability and the statement about $S$-equivalence is proved in a similar way.
\end{proof}

\begin{remark}\ 
\begin{enumerate}
\item As the notation suggests, the set of walls $\tilde{\mathcal W}_K$ depends on the initial choice of $K$.  Clearly, the sets $\mathcal S'_K$ and $\mathcal S_K$ get larger as $K$ gets larger, so $\tilde{\mathcal W}_K$ also gets larger and there is no reason to expect the union over all such $K$ to still be finite.

There is an equivalence relation on $\Amp(X)_{\mathbb R}$ obtained by declaring  $L'$ and $L''$ to be equivalent if all sheaves $E$ of topological type $\tau$ are Gieseker-(semi)stable with respect to $L'$ if and only if this holds with respect $L''$.    Then, the above says that, when restricted to $\Cone(K)$, an equivalence class is given by a union of chambers associated with $\tilde{\mathcal W}_K$.   
\item If one instead only considers the non-trivial walls of the form $\{\tilde{W}_{F,1}\}$ as $F\in \mathcal S_K$, then one gets a coarser chamber structure that witnesses the change in slope stability as $L$ varies in $\Cone(K)$.

\item When discussing threefolds we will call a wall of the form $\tilde{W}_{F,1}$ a \emph{wall of the first kind} (which in this case is a quadratic wall) and a wall of the form $\tilde{W}_{F,2}$ a \emph{wall of the second kind} (which in this case is a linear wall).
\end{enumerate}
\end{remark}

\begin{remark}
 We used  above that the set of sheaves that are slope semistable with respect to some $\gamma$ varying in a given compact subset of $\Pos(X)_{\R}$ is bounded.  As far as we are aware, the first proof of a statement of this form appears in \cite{GrebToma} (an improved proof appears in \Iref{Theorem}{thm:K^+}).  It is because of this that we require $X$ to be smooth and the sheaves in question to be torsion-free.  It would be interesting to know whether such a chamber structure witnessing the change in Gieseker-stability exists more generally for pure sheaves or singular $X$. The next proposition shows that it does in another extreme case, namely that of one-dimensional sheaves.
\end{remark}

\begin{proposition}\label{prop:onedimensionalsheaves}
Let $X$ be a projective manifold of dimension at least two, let $\tau$ be a topological type of one dimensional coherent sheaves on $X$ and fix $K\subset \Amp(X)_{\R}$ a relatively compact open subset of $\Amp(X)_{\R}$.   Then,
\begin{enumerate}
\item The set of pure sheaves of type $\tau$ which are  (Gieseker)-semistable with respect to some polarisation from $K$ is bounded.
\item There exists a finite linear chamber structure on $\Cone(K)$ that witnesses the change in stability for sheaves of topological type $\tau$.
\end{enumerate}
\end{proposition}
\begin{proof}
We first prove boundedness. 
If $E$ is a one-dimensional sheaf on $X$ and $L\in \Amp(X)_\R$, the corresponding Hilbert polynomial consists of only two terms:
$$ P^L_E(m)=\alpha_1^L(E)m+\alpha_0^L(E).$$
Note that $\alpha_0^L(E)$ does not depend on $L$ so we'll just write $\alpha_0(E):=\alpha_0^L(E)$. To $E$ we associate its (effective) one-dimensional algebraic support cycle which we denote by $Z(E)$ as in \cite[Example 15.1.5]{F}.  Then, $\alpha_1^L(E)=L.Z(E)$, which we shall write as $\vol_LZ(E):=L.Z(E)$.

Now fix $H\in\Amp(X)_\Z$ some reference polarisation. As $K$ is relatively compact, there exist positive constants $C_1, \ C_2$ allowing to compare the volumes, i.e., such that
$$C_1 H< L< C_2 H \text{ for all }L\in K.$$
Here, the inequality denotes the order relation induced by the cone $\Amp(X)_\R$ on $N^1(X)_\R$.
Suppose now that $E$ is semistable with respect to some $L\in K$ and let $F$ be any nontrivial coherent subsheaf of $ E$.  
Then, $\alpha_0(F)\vol_LZ(E)\le \alpha_0(E) \vol_LZ(F)$, hence
$$\alpha_0(F) \le \frac{\alpha_0(E) \vol_LZ(F)}{\vol_LZ(E)}\le 
\frac{\max(0,\alpha_0(E)) C_2 \vol_HZ(F)}{ C_1 \vol_HZ(E)}.$$ This implies the upper bound
$$\frac{\alpha_0(F)}{\vol_HZ(F)}\le\frac{C_2\max(0,\alpha_0(E))}{ C_1 \vol_HZ(E)}$$ for the slope of $F$ with respect to $H$. We thus obtain boundedness by  Grothendieck's Lemma \cite[Lemma 1.7.9]{Bible}. 

Once boundedness is established, the proof of the existence and finiteness of the chamber structure parallels the previous one, linearity being obvious.
\end{proof}

\section{Mumford-Thaddeus principle for general line bundles}\label{sec:variationgeneral}

In \cite{GRTI} we proved that any two moduli spaces of Gieseker-semistable sheaves on a smooth projective threefold are connected by a finite number of Thaddeus-flips. In this section, we will extend this structure result to base manifolds of arbitrary dimensions. However, for this we have to restrict to general classes in $N_{1}(X)_{\mathbb R}$, which we now make precise.

\begin{definition}\label{def:generalpolarisation}
  We say that $L\in \Amp(X)_{\mathbb R}$ is \emph{general} if there exists an open relatively compact $K\subset \Amp(X)_{\mathbb R}$ with $L\in \Cone(K)$ such that $L$ does not lie on any of the walls in $\tilde{\mathcal W}_K$.
\end{definition}

As is clear from the definition, choosing such a $K$, the set of points in $\Cone(K)$ that are general are dense in $\Cone(K)$.  In fact they consist of the complement of a finite number of algebraic varieties, from which one sees that in fact the points $\Amp(X)_{\Q}$ that are general are dense in $\Amp(X)_{\R}$.   Observe finally that if $L$ is general then so is $L^r$ for all $r\in \mathbb R_{>0}$.

Now let $K\subset \Amp(X)_{\R}$ be open and relatively compact, and recall the sets $\mathcal S'_{K}$ and $\mathcal S_{K}$ were defined in \eqref{eq:mathcalS'K} and \eqref{eq:mathcalSK}, respectively.

\begin{lemma}\label{lem:mathcalSrelevance}
Let $\sigma' = (\underline{L}',\sigma'_1,\ldots,\sigma'_{j_0})$ and $\sigma'' = (\underline{L}'',\sigma_1'',\ldots,\sigma_{j_0}'')$ be stability parameters.  Set
$$ \gamma' := \sum\nolimits_j \sigma'_j c_1(L_j')^{d-1} \text{ and } \gamma'' := \sum\nolimits_j \sigma''_j c_1(L_j'')^{d-1}$$
and suppose that $\gamma'$ and $\gamma''$ both lie in $p(\Cone(K))$.  Suppose in addition that for any torsion-free sheaf $E$ of topological type $\tau$ and any saturated $F\subset E$ with $F\in \mathcal S_K$ it holds that
$$p_F^{\sigma'}(\le) p_E^{\sigma'} \text{ if and only if } p_F^{\sigma''}(\le) p_E^{\sigma''}$$
Then, such a sheaf $E$ is (semi)stable with respect to $\sigma'$ if and only if it is (semi)stable with respect to $\sigma''$.
\end{lemma}
\begin{proof}
Suppose $E$ is semistable with respect to $\sigma'$.  Then, it is slope semistable with respect to $\gamma'$ which says $E\in \mathcal S'_K$.  To test for (semi)stability of $E$ with respect to $\sigma''$ it is sufficient to consider only saturated subsheaves $F$ of $E$.   If $\mu_{\gamma''}(F) <\mu_{\gamma''}(E)$, then clearly $F$ does not destabilise $E$ with respect to $\sigma''$ (cf.~ \Iref{Example}{ex:RRochII}, or Remark \ref{rem:RRochIII}).  Otherwise $\mu_{\gamma''}(F)\ge \mu_{\gamma''}(E)$  and so $F\in \mathcal S_K$ and we are done.
\end{proof}

\begin{proposition}\label{prop:keymultigeiseker}
Let $L',L''\in \Amp(X)$ and suppose that $L'$ is general.   Consider the family of stability parameters
\begin{equation}
\sigma(t) := \left(L',L'',\frac{1-t}{\vol(L')},\frac{t}{\vol(L'')}\right)\text{ for }t\in [0,1].\label{eq:segment}
\end{equation}
Then, for all $t>0$ sufficiently small
\begin{enumerate}
\item Any sheaf of topological type $\tau$ is Gieseker-(semi)stable with respect to $L'$ if and only if it is $\sigma(t)$-(semi)stable
\item Suppose $E$ and $E'$ are sheaves of topological type $\tau$ that are Gieseker-semistable with respect to $L'$. Then, $E$ and $E'$ are $S$-equivalent in terms of Gieseker-stability defined using $L'$ if and only if they are $S$-equivalent in terms of multi-Gieseker-stability defined using $\sigma(t)$.
\end{enumerate}
Thus, for $t>0$ sufficiently small the stability parameter $\sigma(t)$ is bounded, and $\mathcal M_{L'} = \mathcal M_{\sigma(t)}$.  
\end{proposition}
\begin{proof}
As $L'$ is general, there exists an open relatively compact $K\subset \Amp(X)_{\mathbb R}$ with $L'\in \Cone(K)$ and $L'$ not lying on any wall in $\tilde{\mathcal W}_K$.

Let $E$ be a sheaf of topological type $\tau$.   If $F\subset E$ is saturated with $F\in \mathcal S_K$ we write the difference of reduced multi-Hilbert-polynomials as
$$p_F^{\sigma(t)}(k) - p_E^{\sigma(t)}(k) = \sum_{i=1}^{d} h_{F,i}(t) \frac{k^{d-i}}{(d-i)!}.$$
Clearly $h_{F,i}$ depends only on the topological type of $F$.  Moreover, since for any sheaf $F$ we have $p_F^{\sigma(0)}(k) = \frac{1}{\rank(F)\vol(L')} \chi(F\otimes L'^k) = p^{L'}_F(k)$, it holds that $h_{F,i}(0)$ is a positive multiple of $\beta_{F,i}^{L'}$ (as defined in \eqref{eq:definitionbeta}), and similarly $h_{F,i}(1)$ and $\beta_{F,i}^{L''}$.   In fact, using the Riemann-Roch theorem (cf.\ \Iref{Example}{ex:RRochII} or Remark \ref{rem:RRochIII}) one checks easily that $h_{F,i}\colon [0,1]\to \mathbb R$ is linear.  

We claim that for each $i=1,\ldots,d$ and $F\in \mathcal S_K$ either $h_{F,i}(t)=0$ for all $t\in [0,1]$, or $h_{F,i}(0)\neq 0$.  This is clear when $i=d$ for then $h_{F,i}$ is independent of $t$, so assume $1\le i\le d-1$ and $h_{F,i}(0)=0$.  So $\beta_{F,i}^{L'}=0$, so $L'$ lies on the wall $\tilde{W}_{F,i}$.  So as $L'$ is general this implies that $\tilde{W}_{F,i}$ is trivial, i.e. is the whole of $\Amp(X)_{\mathbb R}$.  In particular $L''\in \tilde{W}_{F,i}$, and thus $h_{F,i}(1) =0$ as well, and so by linearity $h_{F,i}(t) =0$ for all $t\in [0,1]$, thus proving the claim.

Now, as there are only a finite number of topological types among the sheaves in $\mathcal S_K$, for all $t_0>0$ sufficiently small the following holds: for all $F\in \mathcal S_K$ and all $i=1,\ldots,d$ either
\begin{equation}\label{eq:alternative}
h_{F,i}\text{ has no roots in }[0,t_0]\text{ or  }h_{F,i}(t)=0\text{ for all }t\in [0,1].
\end{equation}

We now prove (1) by showing that $E$ is (semi)stable with respect to $\sigma(0)$ (which is equivalent to being Gieseker-(semi)stable with respect to $L'$) if and only if it is (semi)stable with respect to $\sigma(t_0)$.  To this end, let 
$$ \gamma_t = \frac{1-t}{\vol(L')} c_1(L')^{d-1} + \frac{t}{\vol(L'')} c_1(L'')^{d-1} \in N_1(X)_\mathbb{R} .$$
As $L'\in \Cone(K)$, we have $\gamma_{0}\in p(\Cone(K))$ and by continuity, by shrinking $t_{0}$ if necessary, we have $\gamma_{t_{0}}\in p(\Cone(K))$ as well.   So we are in a position to apply Lemma \ref{lem:mathcalSrelevance}, so it is sufficient to prove that for all saturated $F\subset E$ with $F\in \mathcal S_K$ we have
$$ p_F^{\sigma(0)} (\le) p_E^{\sigma(0)} \text{ if and only if }  p_F^{\sigma(t_0)} (\le) p_E^{\sigma(t_0)}.$$
But precisely from the ordering of polynomials (which is the same as the lexicographic ordering on its coefficients) this follows immediately from \eqref{eq:alternative}, proving statement (1) in the theorem.  The statement (2) about $S$-equivalence then follows from this.

Finally, as is well known, the set of sheaves of a given topological type that are semistable with respect to $L'$ is bounded, by \cite[3.3.7]{Bible}, or  \Iref{Theorem}{thm:positiveimpliesstronglybounded}.  Thus, for $t>0$ sufficiently small, $\sigma(t)$ is also bounded, and the final statement follows.
\end{proof}

\begin{theorem}[Mumford-Thaddeus principle for Gieseker moduli spaces]\label{thm:MTprincipleGieseker}
Let $X$ be a projective manifold of dimension $d$ over an algebraically closed field of characteristic zero, let $\tau \in B(X)_\mathbb{Q}$, and let  $L', L''\in \Amp(X)_{\mathbb Q}$ be general. Then, the moduli spaces $\mathcal{M}_{L'}$ and $\mathcal{M}_{L''}$ of torsion-free sheaves of topological type $\tau$ that are Gieseker-semistable with respect to $L'$ and $L''$, respectively, are related by a finite number of Thaddeus-flips.
\end{theorem}

\begin{remark}
  The obvious approach to prove such a statement is to first scale $L'$
  and $L''$ so they are integral, and then consider the two stability
  parameters
  \begin{equation}\label{eq:obviousstabilityparameters}
    \sigma = (L', L'',1,0) \text{ and } \sigma' = (L', L'', 0,1).
  \end{equation}
  Then, Gieseker-(semi)stability with respect to $L'$ (resp. $L''$) is
  precisely multi-Gieseker-(semi)\-stability with respect to $\sigma'$
  (resp.\ $\sigma''$).  Ideally, we would then apply our main variation
  result (\!\!\!\Iref{Corollary}{cor:masterspace}); however, we are prevented
  from doing this since the stability parameters in
  \eqref{eq:obviousstabilityparameters} are clearly not positive (in the sense of \Iref{Definition}{def:stabilityparameter}).  In
  a sequel to this paper we will address this issue, and show how this
  direct naive approach can be indeed carried out to prove that the moduli space of Gieseker-semistable
  sheaves with respect to any two polarisations are related by a
  finite number of Thaddeus-flips.  For now, we use what we have
  already to prove the above weaker version of this statement.
\end{remark}

\begin{proof}[Proof of Theorem \ref{thm:MTprincipleGieseker}]
Consider the segment of stability parameters given by
$$ \sigma(t) = (L',L'',\frac{1-t}{\vol(L')},\frac{t}{\vol(L'')}) \text{ for } t\in [0,1].$$

As $L'$ and $L''$ are general, two applications of Proposition \ref{prop:keymultigeiseker} imply that for $t>0$ sufficiently small $\mathcal M_{L'}  = \mathcal M_{\sigma(t)}$ and $\mathcal M_{L''} = \mathcal M_{\sigma(1-t)}$.    Moreover, as $t\in (0,1)$, both $\sigma(t)$ and $\sigma(1-t)$ are clearly positive stability parameters, and so \Iref{Theorem}{thm:masterspace} applies to give the result.
\end{proof}

\section{Gieseker-stability with respect to real ample classes}

We digress to revisit the moduli spaces $\mathcal M_{\omega}$ of Gieseker-semistable sheaves taken with respect to a real class $\omega\in \Amp(X)_{\mathbb R}$.  Suppose we again fix an open, relatively compact $K\subset \Amp(X)_{\mathbb R}$ and assume that $\omega \in \Cone(K)$.  If $\omega$ is general, so does not lie on any of the walls in $\tilde{\mathcal W}_K$, then we can perturb it slightly to find a rational $L\in \Amp(X)_{\mathbb Q}$ in the same chamber.  Thus, Gieseker-stability with respect to $\omega$ is the same as that for $L$, and so the moduli spaces $\mathcal M_{\omega} = \mathcal M_L$ agree, and so in particular $\mathcal M_{\omega}$ is projective.  

We now show that we can still get such a projective moduli space if $\omega$ is allowed to lie on walls of the form $\tilde{W}_{1,F}$ for $F\in \mathcal S_K$, but not on any walls of the form $\tilde{W}_{i,F}$ for $F\in \mathcal S_K$ and $i\ge 2$.

\begin{theorem}[Projective moduli spaces for $\omega$-semistable sheaves]\label{thm:kaehlermoduli:revisited}
Let $X$ be a smooth projective manifold of dimension $d$ and $\tau \in B(X)_{\mathbb Q}$.  Suppose that $K\subset \Amp(X)_{\mathbb R}$ is open and relatively compact and that $\omega\in \Cone(K)$ does not lie on any of the walls $\tilde{W}_{i,F}$ for $i\ge 2$ and $F\in \mathcal S_K$. 

Then, there exists a projective moduli space $\mathcal M_{\omega}$ of torsion-free sheaves of topological type $\tau$ that are Gieseker-semistable with respect to $\omega$.  This moduli space contains an open set consisting of points representing isomorphism classes of stable sheaves, while points on the boundary correspond to $S$-equivalence classes of strictly semistable sheaves.  
\end{theorem}
\begin{proof}
Clearly $\omega^{n-1}\in \Pos(X)_{\mathbb R}$.  As the map $p\colon \Amp(X)_{\mathbb R}\to \Pos(X)_{\mathbb R}$ is a homeomorphism we can find $L_1,\ldots,L_{j_0}\in \Amp(X)_{\mathbb Q}$ arbitrarily close to $\omega$, so that $\omega^{n-1}$ is a convex combination of $p(L_1),\ldots,p(L_{j_0})$,  say $\omega^{n-1} = \sum_j \sigma_j L_j^{d-1}$.   By rescaling all the $L_j$ simultaneously we may assume all of them to be integral.  We then set $\sigma = (L_1,\ldots, L_{j_0}, \sigma_1,\ldots,\sigma_{j_0})$.

As $\omega$ does not lie on any of the walls $\tilde{W}_{i,F}$ for $i\ge 2$, we can take such $L_j$ close enough to $\omega$ to ensure that $\beta_{i,F}^{L}$ has the same sign as $\beta_{i,F}^{\omega}$ for all $i\ge 2$ and $F\in \mathcal S_K$. Using convexity, one checks that this implies that $p_{F}^{\omega} (\le) p_F^{\omega}$ if and only if $p_{F}^{\sigma}(\le) p_E^{\sigma}$.  So $E$ is Gieseker-(semi)stable with respect to $\omega$ if and only if it is (semi)stable with respect to $\sigma$. Now as the wall structure for multi-Gieseker stability is rational linear \Iref{Proposition}{prop:weakchamberstructure} we can perturb the $\sigma_j$ to be rational without changing stability, and thus $\mathcal M_{\omega} = \mathcal M_{\sigma}$.
\end{proof}

\part{Construction of Uniform Stability Segments}\label{part:examples}

\section{Surfaces}\label{sec:exsurfaces}
We sketch the case of the variation problem when $X$ is a smooth complex surface.  
The upshot will be that we recover, in a natural way, the main result of Matsuki-Wentworth \cite{MatsukiWentworth}.  

\begin{theorem}[Mumford-Thaddeus-principle through moduli spaces of sheaves for surfaces]\label{thm:2intermediate}
Suppose $X$ is a smooth surface and $\tau\in B(X)_{\mathbb Q}$ and let $L,L'\in \Amp(X)$.   Then, the moduli spaces $\mathcal M_{L}$ and $\mathcal M_{L'}$ of torsion-free sheaves of topological type $\tau$ that are Gieseker-semistable with respect to $L$ and $L'$, respectively, are related by a finite number of Thaddeus-flips through moduli spaces of (twisted) semistable sheaves.
\end{theorem}

Throughout this discussion we fix the topological type of the sheaves under consideration.  
The ample cone $\Amp(X)$ is divided into chambers by a locally finite collection of rational linear walls 
that witness the change in slope stability.   Suppose that $L_0,L_1\in \Amp_\Z(X)$ lie in the interior of adjacent chambers separated only by a single wall, so for each wall $W$ the straight line through $L_{0}$ and $L_{1}$ is either contained in $W$ or meets it at a single point (necessarily rational) which we denote by $\overline{L}$.  We may assume that $\overline{L}$ is the midpoint of this segment, and by rescaling we may as well take it to be integral.    For a large positive integer $a$, consider the stability segment given by
$$ \sigma(t) = (\overline{L},\overline{L}; \frac{t}{\vol(\overline{L})} L_1^a , \frac{1-t}{\vol(\overline{L})}L_0^a), \quad t \in [0,1].$$
Using Riemann-Roch one checks easily that $(\sigma(t))_{t \in [0,1]}$ is a uniform segment of stability parameters.  
Moreover, arguing along the lines of \cite[Sect.~3]{MatsukiWentworth}, it is not hard to show 
that for $a$ sufficiently large, any sheaf $E$ is (semi)stable with respect to $\sigma(0)$ (resp.\ $\sigma(1), \sigma(1/2)$) 
if and only if it is Gieseker-(semi)stable with respect to $L_0$ (resp.\ $L_1$, $\overline{L}$);  we omit the details, since the same technique is used more generally below.

Then, we may conclude from Theorem \ref{thm:intermediate} that $\mathcal M_{L_0}, \mathcal M_{L_1}$, and $\mathcal M_{\overline{L}}$ are  related to each other through a finite number of Thaddeus-flips through moduli spaces of sheaves. Since this holds for any two such $L_{0}$ and $L_{1}$, one may apply this statement repeatedly to deduce the same holds for any two moduli spaces of Gieseker-semistable torsion-free sheaves on a smooth surface.

\section{Threefolds}\label{sect:threefolds}

Our aim is to prove the following.

\begin{theorem}[Identification of intermediate spaces]\label{thm:3intermediate}
  Let $X$ be a smooth projective manifold of dimension $3$.    Suppose $L_0$ and $L_1$ are ample line bundles ``separated by a single wall of the first kind''.  Then, the moduli spaces $\mathcal M_{L_0}$ and $\mathcal M_{L_1}$ of Gieseker-semistable torsion-free sheaves of topological type $\tau$ are related by a finite number of Thaddeus-flips through spaces of the form $\mathcal M_{\sigma}$ for some bounded stability parameter $\sigma$.
\end{theorem}

The definition of what it means to be ``separated by a single wall of the first kind'' will be given below (Definition \ref{def:separatedwallfirstkind}).

\begin{remark}\label{rmk:Schmitt}
  Schmitt \cite{Schmitt} essentially considers the case 
that $L_0$ and $L_1$ are separated by a single wall (on a manifold of any dimension), 
and studies the variation problem under the additional assumption 
that this wall contains a rational point $\overline{L}$ (and makes no attempt to identify the intermediate spaces).  The main advantage of the technique used in the present paper is that we do not need to consider points on this wall in $\Amp_\R(X)$, and instead move to the space of multi-Gieseker-stability parameters for the variation problem, at which point the rationality of the wall in $\Amp_\R(X)$ becomes irrelevant.  
\end{remark}

\subsection{Notation and setup for the proof of Theorem~\ref{thm:3intermediate}}

  In order to shorten notation, from here on, a given stability segment $(\sigma(t))_{t \in [0,1]}$ will be written as $\sigma(\cdot)$.  As we are interested in the moduli spaces, rather than the precise stability parameters, 
the following equivalence relations are convenient:

\begin{definition}\label{def:equivalenerelationsstabparameters}
  Let $\sigma$ and $\tilde{\sigma}$ be bounded stability parameters. 
  \begin{enumerate}
  \item  We write $\sigma\equiv \tilde{\sigma}$ if any sheaf of topological type $\tau$ is (semi)stable with respect to $\sigma$ if and only if it is (semi)stable with respect to $\tilde{\sigma}$.    
  \item We write $\mathcal M_{\sigma}  \leftrightdash \mathcal M_{\tilde{\sigma}}$ 
if the moduli spaces $\mathcal M_{\sigma}$ and $\mathcal M_{\tilde{\sigma}}$ are related 
by a finite number of Thaddeus-flips through spaces of the form $\mathcal M_{\sigma_i}$ 
for some bounded stability parameters $\sigma_i$.
  \end{enumerate}
\end{definition}

\subsubsection{Riemann-Roch}
To simplify various Riemann-Roch calculations, 
we shall use the following notation adapted from Schmitt \cite[Sect.~2.1]{Schmitt}.  
Let $X$ be a smooth manifold of dimension $d$ and $L$ be an ample line bundle on $X$.  
 For a torsion-free sheaf $E$ define
$$ \Hilb_i(E) := (\ch(E) \Todd(X))_{i} \quad \text{ and } \quad\hilb_i(E) := \frac{\Hilb_i(E)}{\rk(E)},$$
where the subscript $i$ denotes the part of this product in $H^{2i}(X,\mathbb Q)$.  Also, given a proper subsheaf $F\subset E$ we set
$$ \hilb_i(F,E) := \hilb_i(F) - \hilb_i(E).$$
For top degree forms we shall omit the integration over $X$ when it is clear from context, so for example
$$ \hilb_i(E)c_1(L)^{d-i} : =\int_X \hilb_i(E) c_1(L)^{d-i} = \frac{1}{\rk(E)} \int_X \Hilb_i(E) c_1(L)^{d-i}.$$
Now, for a polynomial of the form
$$p(k) = p_{1} \frac{k^{d-1}}{(d-1)!} + p_2 \frac{k^{d-2}}{(d-2)!} + \cdots + p_d, $$ we write
$$p = \langle\langle  p_{1}||p_2||\cdots||p_d\rangle\rangle$$
for the vector of coefficients.  Thus, the ordering on such polynomials $p$ is the lexicographic order on the vector $\langle\langle  p_{1}||\cdots||p_d\rangle\rangle$.   So, by Riemann-Roch, the reduced Hilbert polynomial of a torsion-free sheaf $E$ with respect to $L$ is $p^L_E = \frac{k^d}{d!} + p(k)$, where
\begin{equation}
p = \frac{1}{\vol(L)}\langle\langle \hilb_1(E)
  c_1(L)^{d-1}||\hilb_2(E)
  c_1(L)^{d-2}||\ldots||\hilb_d(E)\rangle\rangle.\label{eq:rrochinter}
\end{equation}
Finally if $\sigma$ is a stability parameter, then for any proper subsheaf $F\subset E$ we let
$$ p^\sigma_{F\subset E} : = p^\sigma_F - p^\sigma_E.$$

\subsubsection{Gieseker Walls}\label{subsubsect:GiesekerWalls} We now define what we mean for two line bundles to be separated by a single wall of the first kind.  For simplicity here we assume $X$ is smooth of dimension 3 (but a similar story holds in any dimension).  Fix a relatively compact open and connected $K\subset \Amp(X)_{\mathbb R}$.    As defined in Section \ref{sec:chambrstructuregieseker} this gives rise to a collection  $\tilde{\mathcal W}_K$ of non-trivial walls that witnesses the change in Gieseker-stability.   For convenience we recall the construction.  Let $\Cone(K) = \{ \lambda L : L\in K, \lambda\in \mathbb R_{>0}\}$,
\begin{equation}\label{eq:mathcalS'K:repeat}
  \mathcal S'_K := \left\{E \left|\begin{split}  & E \text{ is torsion-free of topological type } \tau \text{ that is }\\
& \text{slope semistable with respect to some } \gamma\in p(\Cone(K))
\end{split} \right.\right\},
\end{equation}
and
\begin{equation}\label{eq:mathcalSK:repeat}
  \mathcal S_K := \left\{
F \left|\begin{split} & F \text{ is a saturated subsheaf of some } E\in \mathcal S'_K \text{ and } \\
&\mu^L(F) \ge \mu^L(E) \text{ for some } L\in \Cone(K)
\end{split}\right.\right\},
\end{equation}
which are both bounded.  For each $F\in \mathcal S_{K}$ write the difference of reduced Hilbert-polynomials as
$$p_{F}^L(k) - p_E^L(k)  =\frac{1}{\vol(L)}\sum_{i=1}^{3} \beta_{F,i}^L \frac{k^{d-i}}{(d-i)!}$$
where $E$ is any sheaf of topological type $\tau$, and for $i=1,2$ let
$$ \tilde{W}_{F,i} = \{ L\in \Amp(X)_{\mathbb R} \mid \beta_{F,i}^L = 0\}.$$

\begin{definition}\label{def:kindsofwalls}
We call the $\tilde{W}_{F,1}$ \emph{walls of the first kind} and the $\tilde{W}_{F,2}$ \emph{walls of the second kind}.  We let $\tilde{\mathcal W}_K$ be the finite set of walls $\tilde{W}_{F,i}$ for $F\in \mathcal S_K$ and $i=1,2$ such that $\tilde{W}_{F,i}$ is neither empty nor all of $\Amp(X)_{\mathbb R}$.
\end{definition}

    \begin{definition}\label{def:separatedwallfirstkind}
Let $K\subset \Amp(X)_{\R}$ be open and relatively compact.     We say that $L_0,L_1\in \Cone(K)\cap \Amp(X)_{\mathbb Q}$ are \emph{separated by a single wall the first kind} if
      \begin{enumerate}
      \item $L_0$ and $L_1$ do not lie on any of the walls in $\tilde{\mathcal W}_K$
      \item The straight line segment between $L_0$ and $L_1$ is contained in $\Cone(K)$ and meets precisely one wall of the first kind, and does not meet any walls of the second kind.
      \end{enumerate}
    \end{definition}

Observe that this definition depends on the chosen set $K$, and that if $L_0$ and $L_1$ are separated by a single wall of the first kind then they are certainly general (as in Definition \ref{def:generalpolarisation}).  Thus, as long as there exist walls of the first kind, there will be plenty of $L_0$ and $L_1$ that satisfy this condition.

\begin{remark}\label{rmk:whereseparatedfirstkindisused}
 The condition imposed in Theorem~\ref{thm:3intermediate} that the two lines bundles be separated by a single wall of the first kind is used only in the final step of the argument, cf.~Lemma~\ref{lem:step3open}.
\end{remark}

\subsubsection{Multi-Gieseker Walls}\label{sec:multigiesekerwalls}

We now describe a wall structure associated to a stability segment $\sigma(\cdot)$.  This is essentially a repeat of our multi-Gieseker chamber structure from \Iref{Section}{sec:chamber}, but we will need something slightly more technical, so we give a self-contained account here.

Suppose that $\sigma(t) = (\underline{L}; B_1(t),\ldots,B_{j_0}(t))$ for $t\in [0,1]$ is a segment of stability parameters.  For each $t\in [0,1]$ set
\begin{equation}
 \gamma_t : = \sum\nolimits_j \rank(B_j(t)) c_1(L_j)^{d-1} \in N_1(X)_\mathbb{R}.\label{eq:defgamma}
\end{equation}
So by Riemann-Roch, if $E$ is semistable with respect to $\sigma(t)$, then it is slope semistable with respect to $\gamma_{t}$.

\begin{definition}\label{def:detectchangemultigeiseker}
We say that a bounded set $\mathcal S$ of sheaves \emph{detects the change of multi-Gieseker-stability for} $\sigma(\cdot)$ if the following holds:  if $E$ of topological type $\tau$ and semistable with respect to $\sigma(t)$ for some $t\in [0,1]$, and $F\subset E$ is saturated with  $\mu_{\gamma_{t'}}(F)\ge \mu_{\gamma_{t'}}(E)$ for some $t'\in [0,1]$, then $F\in \mathcal S$.
\end{definition}

Below we will apply this with $\mathcal S$ being  a set of type $\mathcal S_K$ as in \eqref{eq:mathcalSK:repeat}, but that is not important yet.  The terminology is justified by the following:

\begin{lemma}\label{lem:detectchangemultigieseker}
Let $\mathcal S$ be a bounded set of sheaves that detects the change of multi-Gieseker-stability for a stability segment $\sigma(\cdot)$ and let $t,t'\in [0,1]$.     Suppose that for all torsion-free sheaves $E$ of topological type $\tau$ and all saturated $F\subset E$ with $F\in \mathcal S$ we have
$$p_{F\subset E}^{\sigma(t)} (\le) 0 \text{ if and only if } p_{F\subset E}^{\sigma(t')} (\le) 0.$$
Then $\sigma(t) \equiv \sigma(t')$.
\end{lemma}
\begin{proof}
We have seen essentially this statement before (see Lemma \ref{lem:mathcalSrelevance}).   Suppose that $E$ is of topological type $\tau$ and semistable with respect to $\sigma(t)$, and we show it is also (semi)stable with respect to $\sigma(t')$.  Let $F\subset E$ be saturated.   If $\mu_{\gamma_{t'}}(F)<\mu_{\gamma_{{t'}}}(E)$ then $F$ clearly does not destabilise $E$ with respect to $\sigma(t')$.  
Otherwise by definition $F\in \mathcal S$, and so the hypothesis that $E$ is semistable with respect to $\sigma(t)$ gives $p_{F\subset E}^{\sigma(t)} (\le) 0$ and hence $p_{F\subset E}^{\sigma(t')} (\le) 0$.
\end{proof}

 Suppose now that $\mathcal S$ detects the change of multi-Gieseker-stability for $\sigma(\cdot)$.  For each $F\in \mathcal S$ write the difference of reduced multi-Hilbert polynomials as
$$ p_F^{\sigma(t)}  - p_E^{\sigma(t)} = \sum_{i=1}^{d} h_{F,i}(t) \frac{k^{d-i}}{(d-i)!},$$
where $E$ is any sheaf of topological type $\tau$.   Then,  $h_{F,i}\colon [0,1]\to \mathbb R$ is linear, and thus is either identically zero or has at most 1 root in $[0,1]$.

\begin{definition}\label{def:multiGiesekerwalls}
Let $\mathcal S$ be a bounded set of sheaves that detects the change of multi-Gieseker-stability for $\sigma(\cdot)$.  We define the set of \emph{multi-Gieseker walls associated to} $\mathcal S$ to be the set of all roots of $h_{F,i}$ for $i=1,\ldots,d$ and $F\in \mathcal S$ among all those $h_{F,i}$ that are not identically zero.
\end{definition}

As each $h_{F,i}$ depends only on the topological type of $F$, the set of multi-Gieseker walls associated to $\mathcal S$ consists of a finite number of rational points in $[0,1]$, which we denote by $\overline{t}_1<\ldots\overline{t}_N$.  This divides $[0,1]$ into ``chambers'' within which multi-Gieseker-stability is unchanged.

\begin{corollary}\label{cor:detectmultiGieseker} Let $\sigma(\cdot)$ be a stability segment and let  $\overline{t}_{1}<\ldots<\overline{t}_{N}$ be the multi-Gieseker walls associated to a bounded set of sheaves that detects the change in multi-Gieseker-stability for $\sigma(\cdot)$.  Then,
$$\mathcal M_{\sigma(t)} = \mathcal M_{\sigma(t')} \text{ for all }t,t'\in (\overline{t}_i,\overline{t}_{i+1}).$$  
\end{corollary}
\begin{proof}
  This is immediate from Lemma \ref{lem:detectchangemultigieseker}, since for all saturated $F\subset E$ with $F\in \mathcal S$ we have $p_{F\subset E}^{\sigma(t)} (\le) 0$ if and only if $p_{F\subset E}^{\sigma(t')} (\le) 0$ by the definition of ordering of polynomials and the definition of the points $\overline{t}_{i}$.
\end{proof}

\subsubsection{Open and (almost) perfect stability parameters}

\begin{definition}\label{def:openstabilitysegment}
  We say the stability segment $\sigma(\cdot)$ is \emph{open} if for all $t>0$ sufficiently small we have $\sigma(0)\equiv \sigma(t)$ and $\sigma(1)\equiv \sigma(1-t)$. 
\end{definition}

The purpose of this definition is to avoid the possibility of non-trivial change of moduli space occurring at the endpoints of $[0,1]$. 

\begin{definition}\label{def:perfect}
Let $\mathcal S$ be a bounded set of sheaves.  We say a stability segment $\sigma(\cdot)$ is \emph{almost perfect} with respect to $\mathcal S$ if the following all hold:
\begin{enumerate}
\item The set $\mathcal S$ detects the change in multi-Gieseker-stability for $\sigma(\cdot)$.
\item $\sigma(\cdot)$ is open.
\item We have  $\mathcal M_{\sigma(t_0)} \leftrightdash \mathcal M_{\sigma(t_1)}$ for all $t_0,t_1\in (0,1)$ that are not among the multi-Gieseker walls associated to $\mathcal S$ (cf.\ Definition \ref{def:equivalenerelationsstabparameters}).
\end{enumerate}
\end{definition}
\begin{remark}\
\begin{enumerate}
\item[i)] We did not make any requirement that we make a ``minimal'' choice of multi-Gieseker walls (so certainly some of the walls could be superfluous).  Thus the condition (3) could, in principle, depend on the choice of set $\mathcal S$.  We observe also that even if $\sigma(\cdot)$ is open, some of these walls could be at the endpoints of $[0,1]$; however we only demand that condition (3) holds for points in the interior of $[0,1]$. 
\item[ii)] By Corollary \ref{cor:detectmultiGieseker}, and since there are only a finite number of multi-Gieseker walls, we can replace condition (3) by requiring instead that  $\mathcal M_{\sigma(t_0)} \leftrightdash \mathcal M_{\sigma(t_1)}$ for general points $t_0,t_1$ in the two open chambers adjacent to a single wall in $(0,1)$.    
\item[iii)] Furthermore,  by the work done in Section~\ref{sect:uniformity}, see especially Theorem \ref{thm:intermediate}, if $\sigma(\cdot)$ is bounded and uniform, then condition (3) certainly holds.  
\item[iv)]  We reserve the terminology \emph{perfect stability segment} to mean that condition (3) holds for all $t_0,t_1\in (0,1)$ including the walls, but this notion will hardly be used.
\end{enumerate}
\end{remark}

The purpose of this definition is the following obvious conclusion.

\begin{lemma}\label{lem:thepointofalmostperfect}
Suppose that $\sigma(\cdot)$ is almost perfect with respect to some bounded set $\mathcal S$ of sheaves.  Then, $\mathcal M_{\sigma(0)} \leftrightdash \mathcal M_{\sigma(1)}$.  
\end{lemma}
\begin{proof}
As $\sigma(\cdot)$ is open, we have for $t>0$ sufficiently small that $\mathcal M_{\sigma(0)} = \mathcal M_{\sigma(t)}$ and $\mathcal M_{\sigma(1)} = \mathcal M_{\sigma(1-t)}$.  Thus, we can apply condition (3) to deduce that $\mathcal M_{\sigma(t)} \leftrightdash \mathcal M_{\sigma(1-t)}$.
\end{proof}

\subsection{Strategy of the proof}\label{sec:strategy}
The idea for the proof of Theorem~\ref{thm:3intermediate} is rather natural: We first consider a stability segment $\sigma(\cdot)$ that joins two stability parameters we are interested in.  In general, it will not be uniform, but it will be cut up by a finite number of rational walls.  We then focus our attention on a specific wall, and find a new stability segment $\eta(\cdot)$ that joins stability parameters immediately to the left and to the right of this wall.  Moreover, we arrange that this new stability segment is closer to being uniform than $\sigma(\cdot)$, in that more terms of the corresponding multi-Hilbert polynomial are independent of $t$. We then apply the same argument to $\eta(\cdot)$, and continuing in this way, we arrive at a uniform stability segment, to get the desired conclusion.  Thus, since we are on a threefold, the proof divides into three steps, whose main arguments we will summarise now.

\subsubsection{Step 1}

Fix an open, relatively compact, connected $K\subset \Amp(X)_{\R}$, and assume $L_0, L_1 \in \Cone(K)$ are separated by a single wall of the first kind.  Without loss of generality we may assume that $L_{0}$ and $L_{1}$ are arbitrarily close to each other, and so we may assume that the line segment between $p(L_{0})$ and $p(L_{1})$ is completely contained in $p(\Cone(K))$.   Moreover, by perturbing $L_0$ and $L_1$ slightly we may, without loss of generality, assume that
\begin{equation}\label{eq:gen}
\begin{aligned}
\int_X c_1(L_0)^3 &\neq \int_X c_1(L_1)c_1(L_0)^2\quad \quad\text{ and}\\
\int_X c_1(L_1)^3 &\neq \int_X c_1(L_0)c_1(L_1)^2.
\end{aligned}
\end{equation}
Finally, by rescaling both $L_0$ and $L_1$ simultaneously we may assume they are both integral.   As we have done in Section~\ref{subsubsect:GiesekerWalls} above, set
\begin{equation}\label{eq:mathcalS'_K:repeatagain}
  \mathcal S'_K := \left\{E \left|\begin{split}  & E \text{ is torsion-free of topological type } \tau \text{ that is }\\
& \text{slope semistable with respect to some } \gamma\in p(\Cone(K))
\end{split} \right.\right\},
\end{equation}
and
\begin{equation}\label{eq:mathcalS_K:repeatagain}
  \mathcal S_K := \left\{
F \left|\begin{split} & F \text{ is a saturated subsheaf of some } E\in \mathcal S'_K \text{ and } \\
&\mu^L(F) \ge \mu^L(E) \text{ for some } L\in \Cone(K)
\end{split}\right.\right\}
\end{equation}
Now set
\begin{equation*}
 \sigma_0(t) := \frac{1-t}{\vol(L_0)} , \quad \quad
\sigma_1(t) = \frac{t}{\vol(L_1)}
\end{equation*}
and consider the stability segment $$ \sigma(t) := (L_0,L_1, \sigma_0(t), \sigma_1(t)).$$
Clearly, a sheaf is (semi)stable with respect to $\sigma(0)$ (resp.\ $\sigma(1)$) if and only if it is (semi)stable with respect to $L_0$ (resp.\ $L_1$).    Thus, to prove Theorem~\ref{thm:3intermediate} it is sufficient to prove that $\sigma(\cdot)$ is almost perfect with respect to $\mathcal S_K$, for then $\mathcal M_{L_0} = \mathcal M_{\sigma(0)}\leftrightdash \mathcal M_{\sigma(1)}= \mathcal M_{L_1}$.     To this end, we start with:

\begin{lemma}\label{lem:S_kwitnesssigma}
$\mathcal S_K$ detects the change of multi-Gieseker-stability for $\sigma(\cdot)$.
\end{lemma}

This, and all subsequent lemmata  will be proved Section~\ref{subsect:proofslemmata} below.   The assumption that $L_0$ and $L_1$ are separated by a single wall of the first kind in particular implies they are general.  Thus, by Proposition \ref{prop:keymultigeiseker} we conclude that $\sigma(\cdot)$ is an open segment of stability parameters.   

 Now, since there is no reason to expect $\sigma(\cdot)$ is uniform,    we have to look more closely at what happens when a wall is crossed.    So, fix a wall $\overline{t}\in (0,1)$ associated to $\mathcal S_K$ and let $t_0$ and $t_1$ be general rational points in $(0,1)$ in the chamber immediately to the left and to the right of $\overline{t}$, respectively (the precise requirement for what it means for $t_0$ and $t_1$ to be general can be seen from the proof of Lemma \ref{lem:step2open} below).    So, to show that $\sigma(\cdot)$ is almost perfect with respect to $\mathcal S_{K}$ it remains to show that
\begin{equation}
 \mathcal M_{\sigma(t_0)} \leftrightdash \mathcal M_{\sigma(t_1)}.\label{eq:step1}
\end{equation}

\subsubsection{Step 2}
Fix $\overline{t},t_0,t_1$ as above, and also fix a large positive integer $a$ (to be determined below).  For $s\in [0,1]$ consider the formal sum of line bundles
\begin{equation}\label{eq:defB}
B_j(s) := \sigma_j(\overline{t}) \left( s L_j^{\frac{a\sigma_j(t_1)}{\sigma_j(\overline{t})}} + (1-s)L_j^{\frac{a\sigma_j(t_0)}{\sigma_j(\overline{t})}}\right) \quad j=1,2.
\end{equation}
Observe that as $\overline{t}\in (0,1)$ we have $\sigma_j(\overline{t})\neq 0$; moreover, for \eqref{eq:defB} to make sense we require that $a\sigma_j(t_i)/\sigma_j(\overline{t})\in \mathbb N$ for $i,j=0,1$, which certainly holds for arbitrarily large $a$.
Now define 
$$\eta(s) := (L_0,L_1; B_0(s), B_1(s))\text{ for } s\in [0,1],$$
which we observe is a stability segment, as for all $s\in [0,1]$ we have
\begin{equation*}
  \rank(B_0(s)) \vol(L_0) + \rank(B_1(s))\vol(L_1)= \sigma_{0}(\overline{t}) \vol(L_0) + \sigma_{1}(\overline{t})\vol(L_1)  =1.\label{eq:rankBi}
\end{equation*}

\begin{remark}
It may help the reader to observe, more explicitly, that the multi-Hilbert polynomial of $\eta(s)$ for a sheaf $E$ is given by
\begin{align*}
  P^{\eta(s)}_E(k) &= \frac{1-\overline{t}}{\vol(L_0)}s\chi(E\otimes L_0^k \otimes L_0^{\frac{a(1-t_{1})}{1-\overline{t}}})  + \frac{1-\overline{t}}{\vol(L_0)}(1-s)\chi(E\otimes L_0^k \otimes L_0^{\frac{a(1-t_{0})}{1-\overline{t}}})\nonumber\\
  &\quad +\frac{\overline{t}}{\vol(L_1)}s\chi(E\otimes L_1^k \otimes
  L_1^{\frac{a t_{1}}{\overline{t}}}) +
  \frac{\overline{t}}{\vol(L_1)}(1-s)\chi(E\otimes L_1^k \otimes
  L_1^{\frac{a t_{0}}{\overline{t}}}).
\end{align*}
\end{remark}
There are two reasons for this particular choice of twisting.  First, from Riemann-Roch one can verify that the part of the $k^2$ coefficient of $P_E^{\eta(s)}$ that depends on $E$ is independent of $s$ (this essentially follows as $\rank(B_j(s))$ is independent of $s$).   So, even though $\eta(\cdot)$ may not be a uniform segment of stability parameters,  it is closer to being uniform than $\sigma(\cdot)$ was.  Second, at the endpoints $s=0$ and $s=1$ we get back the stability with respect to $\sigma(t_0)$ and $\sigma(t_1)$, as made precise in the following lemma which is adapted from ideas of Matsuki-Wentworth \cite[Thm 4.1]{MatsukiWentworth}.  

\begin{lemma}\label{lem:step2endpoints}
For $a\in \N$ sufficiently large,  $\eta(0) \equiv \sigma(t_0)$ and $\eta(1) \equiv \sigma(t_1)$.
\end{lemma}

As a consequence, our goal becomes to show that $\eta(\cdot)$ is almost perfect with respect to $\mathcal S_K$,  for then $\mathcal M_{\sigma(t_0)} =\mathcal M_{\eta(0)}\leftrightdash \mathcal M_{\eta(1)} = \mathcal M_{\sigma(t_1)}$.  Towards this goal we start with:

\begin{lemma}\label{lem:step2open}\
\begin{enumerate}
\item $\mathcal S_K$ detects the change of multi-Gieseker-stability for $\eta(\cdot)$. \item For $a\in \mathbb N$ sufficiently large,  $\eta(\cdot)$ is an open segment of stability parameters.
\end{enumerate}
\end{lemma}

Now, fix a multi-Gieseker wall $\overline{s}\in (0,1)$ for $\eta(\cdot)$ associated to $\mathcal S_K$ and let $s_0$ and $s_1$ be points in $(0,1)$ contained in adjacent chambers located immediately either side of this wall.  So, to prove $\eta(\cdot)$ is almost perfect with respect to $\mathcal S_{K}$, it remains to prove 
\begin{equation}
 \mathcal M_{\eta(s_0)} \leftrightdash \mathcal M_{\eta(s_1)}.\label{eq:step2}
\end{equation}

\subsubsection{Step 3}
Consider a fixed $\overline{s}, s_0$ and $s_1$ as above.  The third, and final, stability segment that we need is built from the following elementary, but tedious, construction:

\begin{lemma}\label{lem:finaltwist}
For sufficiently large $\lambda\in \mathbb R$ and all $b\in \mathbb R_{>0}$, for $r\in [0,1]$ there exist formal sums of line bundles $D_1(r)$ and $D_2(r)$ whose coefficients are linear in $r$ and strictly positive for $r\in (0,1)$ such that the following hold:
\begin{enumerate}
\item $\rank(D_j(r)) = \rank (B_j(\overline{s}))$ and $c_1(D_j(r)) = c_1(B_j(\overline{s}))$  (and so in particular both are independent of $r$).
\item For $i=0,1$ we have
$$\sum_{j=0}^1 \ch_2(D_j(i)) = 
b \sum_{j=0}^1 \rank (B_j(\overline{s})) \left(\lambda + \frac{a(s_i-\overline{s})(\sigma_j(t_1)-\sigma_j(t_0))}{\sigma_j(\overline{t})}\right) c_1(L_j)^2 + \sum_{j=0}^1 \ch_2(B_j(\overline{s}))$$
\item The quantity $\sum_{j=0}^1 c_1(L_j) \ch_2(D_j(r))$ is independent of $r$.
\end{enumerate}
\end{lemma}
Here, we have extended the definition of Chern character linearly to formal sum of line bundles; so, if $A_{ji}$ are line bundles, $\ch(\sum_{i} a_{ji} A_{ji}) := \sum_i a_{ji} \ch(A_{ji})\in H^*(X,\mathbb R)$. In particular, $\rank(\sum_i a_{ji} A_{ji}) = \sum_i a_{ji}$ as defined before.

Granted this, consider
$$ \zeta(r) :=  (L_0,L_1; D_0(r), D_1(r)) \text{ for }r\in [0,1].$$
\begin{lemma}\label{lem:step3endpoints}\
  \begin{enumerate}
  \item $\mathcal S_K$ detects the change of multi-Gieseker-stability for $\zeta(\cdot)$,
  \item $\zeta(\cdot)$ is a bounded and uniform segment of stability parameters, and \item for $b\in \R_{>0}$ sufficiently large, $\zeta(0) \equiv \eta(s_0)$ and $\zeta(1) \equiv \eta(s_1)$.
\end{enumerate}
\end{lemma}

\begin{lemma}\label{lem:step3open}
For $b \in \R_{>0}$ sufficiently large, $\zeta(\cdot)$ is open, and thus perfect.
\end{lemma}

Thus, by Theorem \ref{thm:intermediate} the moduli spaces defined by the endpoints of $\zeta$ are related by a finite number of Thaddeus-flips through moduli spaces of sheaves.  That is,
$$ \mathcal M_{\eta(s_0)} = \mathcal M_{\zeta(0)} \leftrightdash \mathcal M_{\zeta(1)} = \mathcal M_{\eta(s_1)}.$$
This proves \eqref{eq:step2}, which in turn proves \eqref{eq:step1}, which finally proves Theorem \ref{thm:3intermediate}.

\subsection{Proofs of the lemmata} \label{subsect:proofslemmata}
In this section, we prove Theorem~\ref{thm:3intermediate} by establishing the lemmata stated in the previous section. The subsequent proofs involve Riemann-Roch computations that for convenience we state here. 

 Suppose that $$\sigma = (L_0,L_1;B_0,B_1),$$
where $B_i$ is a formal sum of line bundles as in \eqref{eq:formalsum} (whose coefficients could depend on a parameter) and assume
$$ \rank(B_0) \vol(L_0) + \rank(B_1) \vol(L_1) =1.$$
Then, by Riemann-Roch, if $E$ is a non-trivial torsion-free sheaf on our threefold $X$, the multiplicity of $E$ is $r^{\sigma}_E = \rank (E)$, and if $F\subset E$ is a non-trivial subsheaf, then
\begin{equation*}
   p^\sigma_{F\subset E} = \langle\langle h_1||h_2||h_3\rangle\rangle,
\end{equation*}
where 
\begin{equation} \label{eq:h1}
\begin{aligned}
h_1 &= \hilb_1(F,E).\sum \rank(B_j) c_1(L_j)^2,\\
h_2 &= \hilb_2(F,E).\sum \rank(B_j)c_1(L_j) +\hilb_1(F,E).\sum c_1(L_j) c_1(B_j),\\
h_3 &= \hilb_3(F,E). \sum \rank(B_j) + \hilb_2(F,E) .\sum c_1(B_j)+ \hilb_1(F,E).\sum \ch_2(B_j),
\end{aligned}
\end{equation}
the sums being over $j=0,1$.

\begin{proof}[Proof of Lemma \ref{lem:S_kwitnesssigma}]
We have to show $\mathcal S_K$ detects the change of multi-Gieseker-stability for $\sigma(\cdot)$.  By definition $\sigma(t) = (L_0,L_1;\sigma_0(t), \sigma_1(t)) = (L_0,L_1;\sigma_0(t)\mathcal O_X, \sigma_1(t) \mathcal O_X)$.  Thus, as in \eqref{eq:defgamma},  we must consider the curve class
\begin{equation}
 \gamma_t = \sigma_0(t) c_1(L_0)^2 + \sigma_1(t) c_1(L_1)^2 \text{ for } t\in [0,1].\label{eq:defgammat}
\end{equation}
Observe that since we assumed the line segment between $p(L_0)$ and $p(L_1)$ lies completely within $p(\Cone(K))$, we have that $\gamma_t\in p(\Cone(K))$ for all $t\in [0,1]$. 

Suppose that $E$ is of topological type $\tau$ and semistable with respect to $\sigma(t)$ for some $t\in [0,1]$.  Then, it is slope semistable with respect to $\gamma_t$, so $E\in \mathcal S'_K$ (as defined in \eqref{eq:mathcalS'_K:repeatagain}).  Suppose $F\subset E$ is saturated and $\mu_{\gamma_{t'}}(F) \ge \mu_{\gamma_{t'}}(E)$ for some $t'\in [0,1]$.  Then, letting $L'\in \Cone(K)$ be such that $p(L') = \gamma_{t'}$ we have $\mu^{L'}(F)\ge \mu^{L'}(E)$.  Thus, $F\in \mathcal S_K$ (as defined in \eqref{eq:mathcalS_K:repeatagain}).  Therefore, $\mathcal S_K$ detects the change in change of multi-Gieseker-stability for $\sigma(\cdot)$ as claimed.
\end{proof}

For later use, we record that if $F\subset E$, then by Riemann-Roch \eqref{eq:h1}
 \begin{equation}\label{eq:coeffsigma}
  p_{F\subset E}^{\sigma(t)} = \langle\langle h_1(t) || h_2(t) ||h_3(t)\rangle \rangle,
 \end{equation}
 where
 \begin{equation} 
  h_i(t) = \hilb_i(F,E) .\sum\nolimits_j \sigma_j(t) c_1(L_j)^{3-i}\label{eq:hit}.
 \end{equation}

\begin{proof}[Proof of Lemma \ref{lem:step2open}]
We have to show (1) $\mathcal S_K$ detects the change of multi-Gieseker-stability for $\eta(\cdot)$ and (2) For $a\in \mathbb N$ sufficiently large,  $\eta(\cdot)$ is an open segment of stability parameters.

The proof of (1) is easy.   We recall by definition, see \eqref{eq:defB}, $\eta(s) = (L_0,L_1; B_0(s), B_1(s))$, where
\begin{equation}
B_j(s) = \sigma_j(\overline{t}) \Bigl( s L_j^{\frac{a\sigma_j(t_1)}{\sigma_j(\overline{t})}} + (1-s)L_j^{\frac{a\sigma_j(t_0)}{\sigma_j(\overline{t})}}\Bigr).
\end{equation}
Hence,
\begin{equation}\label{eq:ranksandcherns}
\begin{aligned}
\rank(B_j(s)) &= \sigma_j(\overline{t}),\\
c_1(B_j(s))&= a( s\sigma_j(t_1) + (1-s)\sigma_j(t_0)) c_1(L_j)\\
&= a \sigma_j(st_1 + (1-s)t_0) c_1(L_j),\\
\ch_2(B_j(s))&= \frac{a^2}{2\sigma_j(\overline{t})} \left(s\sigma_j(t_1)^2+ (1-s) \sigma_j(t_0)^2 \right)c_1(L_j)^2.
\end{aligned}
\end{equation}

Hence, to show $\mathcal S_K$ witnesses the change in multi-Gieseker-stability for $\eta(\cdot)$ we are required to look at the curve class
$$ \gamma_{\overline{t}}:= \sigma_0(\overline{t}) c_1(L_0)^2 +  \sigma_1(\overline{t}) c_1(L_1)^2$$
(which actually is independent of $s$).   Now, just as in the proof of Lemma \ref{lem:S_kwitnesssigma}, in which we showed that $\mathcal S_K$ witnesses the change in multi-Gieseker-stability for $\sigma(\cdot)$, we have $\gamma_{\overline{t}} \in p(\Cone(K))$.  Thus, exactly the same proof shows that $\mathcal S_{K}$ detects the change in multi-Gieseker-stability for $\eta(\cdot)$.  

We now turn to proving that for $a\in \mathbb N$ sufficiently large,  $\eta(\cdot)$ is an open segment of stability parameters.  By Lemma \ref{lem:detectchangemultigieseker} it is sufficient to show that for $s>0$ sufficiently small, if $E$ is torsion-free of topological type $\tau$ and $F\subset E$ is saturated with $F\in \mathcal S_K$, we have
\begin{equation}p_{F\subset E}^{\eta(0)} (\le) 0 \text{ if and only if } p_{F\subset E}^{\eta(s)} (\le)0\label{eq:sufficientetaopen}
\end{equation}
and
\begin{equation*}p_{F\subset E}^{\eta(1)} (\le) 0 \text{ if and only if } p_{F\subset E}^{\eta(1-s)} (\le)0.
\end{equation*}
We will prove the first of these, as the proof of the second is identical.  Let $F\subset E$ be saturated with $F\in \mathcal S_K$ and $E$ torsion-free of topological type $\tau$.  Then, using Riemann-Roch \eqref{eq:h1} we have
\begin{equation}
 p_{F\subset E}^{\eta(s)} = \langle\langle u_1 || u_2(s) || u_3(s)\rangle\rangle,\label{eq:coeffeta}
\end{equation}
where 
\begin{equation}\label{eq:ui}
\begin{aligned}
  u_1 &= h_1(\overline{t}),\\
  u_2(s) &= h_2(\overline{t}) + ah_1(st_1 +(1-s)t_0),\\
  u_3(s) &= h_3(\overline{t}) + a h_2(st_1 + (1-s)t_0) + a^2 \epsilon(s),
\end{aligned}
\end{equation}
with $h_i(\cdot)$ defined as in \eqref{eq:hit} and 
$$\epsilon(s):=\frac{1}{2}\hilb_1(F,E).\sum\nolimits_{j} \frac{s \sigma_j(t_1)^2 + (1-s)\sigma_j(t_0)^2}{\sigma_j(\overline{t})} c_1(L_j)^2.$$ 

We claim that for $i=1,2,3$ we either have $u_i(s)\equiv 0$ for all $s\in [0,1]$ or $u_i(0)\neq 0$.    Observe that this claim proves \eqref{eq:sufficientetaopen} (and thus completes the proof of the lemma), since we can take $s>0$ sufficiently small so there are no roots of $u_i$ in $[0,s]$ for all (non-trivial) such $u_i$, and then \eqref{eq:sufficientetaopen} follows from the definition of ordering of polynomials in terms of the lexicographic order on the coefficients.
 
Now the claim is clear for $u_1$ as it is independent of $s$.  For $u_2$, notice that if $u_2(0)=0$, then $h_2(\overline{t}) + ah_1(t_0)=0$.  So, as $a$ can be chosen arbitrarily large, this implies $h_1(t_0)=0$ and $h_2(\overline{t})=0$.  
However, $t_0$ was chosen to be a point in an open chamber for $\sigma(\cdot)$,
so $h_1(t_0)=0$ implies $h_1\equiv 0$. Looking back at the expression \eqref{eq:ui} for $u_2$,
we then conclude that $u_2$ is actually independent of $s$, which proves the claim for $u_2$.

A similar argument works for $u_3$.  Suppose that 
$$0=u_3(0)=h_3(\overline{t}) + a h_2(t_0) + a^2 \epsilon(0).$$  Then, as $a$ can be chosen arbitrarily large, we must have $\epsilon(0)=0$, which means
$$0 = \hilb_1(F,E) \left( \frac{(1-t_0)^2}{\vol(L_0)(1-\overline{t})} c_1(L_0)^2 + \frac{t_0^2}{\vol(L_1)\overline{t}} c_1(L_1)^2\right). $$
Think of this as a quadratic polynomial in $t_0$, whose linear term is $$-\frac{2}{\vol(L_0)(1-\overline{t})}\hilb_1(F,E) c_1(L_0)^2.$$   Then, by the assumption that $L_0$ does not lie on any wall in $\tilde{\mathcal W}_K$,  either $\hilb_1(F,E)c_1(L)^2=0$ for all $L\in \Amp(X)_{\mathbb R}$ (in which case $\epsilon(\cdot)\equiv 0$), or this linear term is non-trivial, and so $\epsilon(0)$ is non-constant in $t_0$.  Thus, as $t_0$ was assumed to be a general point in the open chamber adjacent to $\overline{t}$, we may as well assume it was chosen so $\epsilon(0)\neq 0$, which is absurd.  Moreover, as there are only a finite number of topological types among all the sheaves $F\subset E$ under consideration, due to boundedness of $\mathcal S_K$, we can make such a choice uniformly over all such pairs.

Thus, we must in fact have  $\epsilon \equiv 0$.  
So, looking back at the expression for $u_3$, the assumption $u_3(0)=0$ is $0=h_3(\overline{t}) + ah_2(t_0)$ and hence $h_2(t_0)=0$ as $a$ is assumed to be arbitrarily large.  As $t_0$ is general in the open chamber, this implies $h_2(t)= 0$ for all $t\in[0,1]$, at which point one observes $u_3$ is also independent of $s$, which proves the claim for $u_3$.
\end{proof}

\begin{proof}[Proof of Lemma \ref{lem:step2endpoints}]
We shall show that $\eta(0)\equiv \sigma(t_0)$ (the other statement being proved similarly).      We first establish the following
\medskip

\noindent\emph {Claim: } It is sufficient to prove that if $E$ is of topological type $\tau$ and $F\subset E$ is saturated with $F\in \mathcal S_{K}$, then
\begin{equation}
p_{F\subset E}^{\sigma(t_0)} (\le) 0 \text{ if and only if } p_{F\subset E}^{\eta(0)} (\le) 0.\label{eq:lemstep2endpointssufficient}
\end{equation}\smallskip

We have seen essentially this argument before: Assume \eqref{eq:lemstep2endpointssufficient} holds and suppose that $E$ is (semi)stable with respect to $\sigma(t_{0})$.  Then, it is slope semistable with respect to $\gamma_{t_{0}}$ \eqref{eq:defgammat}, and so $E\in \mathcal S'_{K}$.  We wish to show that $E$ is (semi)stable with respect to $\eta(0)$. So, suppose there exists a saturated subsheaf $F\subset E$ such that $p_{F\subset E}^{\eta(0)} > 0$.
Then, $F\in\mathcal S_K$ since $\mathcal S_K$ detects the change in multi-Gieseker-stability for $\eta(\cdot)$. Hence, by condition \eqref{eq:lemstep2endpointssufficient} we have $p_{F\subset E}^{\sigma(t_0)} >0$, which is a contradiction.
The converse is proved similarly.

So, we proceed to prove \eqref{eq:lemstep2endpointssufficient} for any such sheaves $F$ and $E$.   From \eqref{eq:coeffsigma}, \eqref{eq:coeffeta}, \eqref{eq:ui} we have
$$p_{F\subset E}^{\sigma(t_0)} = \langle\langle h_1(t_0) || h_2(t_0) || h_3(t_0) \rangle \rangle,$$
and
\begin{align*}
p_{F\subset E}^{\eta(0)} &= \langle\langle u_1 || u_2(0) || u_3(0)\rangle\rangle\\
&=\langle\langle h_1(\overline{t}) || h_2(\overline{t}) + ah_1(t_0) || h_3(\overline{t}) + a h_2(t_0) + \hilb_1(F,E)\cdot \epsilon \rangle \rangle,
\end{align*}
where $h_j(\cdot)$ is as in \eqref{eq:hit}, and $\epsilon$ is a linear combination of $c_1(L_0)^2$ and $c_1(L_1)^2$.  

Assume $p_{F\subset E}^{\eta(0)} (\le)0$ and we shall show that $p_{F\subset E}^{\sigma(t_0)}(\le)0$.  If $h_1(\overline{t})<0$, then $h_1(t_0)<0$,
since we are assuming that $t_0$ lies in an open chamber immediately adjacent to $\overline{t}$. Thus, we also have 
$p_{F\subset E}^{\sigma(t_0)}(\le)0$ in this case. 
If  $h_1(\overline{t})=0$ and $h_2(\overline{t}) + ah_1(t_0)<0$, then by making $a$ sufficiently large we get either
$h_1(t_0)<0$, in which case we are done, or $  h_1(t_0)=0$ and $h_2(\overline{t})<0$, in which case $h_1\equiv0$ and 
$h_2(\overline{t})<0$. 
But then $h_2(t_0)<0$, since $t_0$ lies in an open chamber immediately adjacent to $\overline{t}$ ,
and we are done again.
So, assume now that $h_1\equiv0$, $h_2\equiv0$. Then, $\hilb_{1}(F,E).\epsilon=0$ as well and the implication follows from the fact that $h_3(t)$ has constant sign on $[0,1]$.

For the converse, if $h_1(t_0)<0$, then $h_1(\overline{t})\le 0$ 
since $\overline{t}$ lies in the closure of the chamber containing $t_0$, 
and   $h_2(\overline{t})+ah_1(t_0)(\le) 0$
for  sufficiently large $a$. If $h_1(t_0)=0$. then $h_1\equiv 0$, as $t_0$ is in an open chamber, which in turn
implies $\hilb_{1}(F,E).\epsilon=0$. We deal now with the conditions $h_2(t_0)<0$ and $h_2(t_0)=0$ in the same way we have dealt with
$h_1(t_0)<0$ and $h_1(t_0)=0$.   Observe that the size of $a$ needed depends only on the topological type of $F$, and since there are only finitely many types of $F$ under consideration we can pick such an integer $a$ uniformly.
\end{proof}

\begin{remark}\label{rmk:simplified}
  The previous two proofs can be simplified at a cost of making the stability segment more complicated.  More precisely, if we arrange for the $B_j(s)$ to have the same rank and first Chern class as before, but with $\ch_2(B_j) = \lambda \sigma_j(\overline{t})$ for some $\lambda\in \mathbb R$.  Then $\epsilon \equiv \lambda h_1(\overline{t})$ (so in particular is constant in $s$).    We leave it as an exercise to the reader to prove that one can construct such $B_j(s)$ of the form $\frac{1}{i_0} \sigma_j(\overline{t}) \sum_{i=1}^{i_0} L_j^{a_{ij}}$ for suitable  $a_{ij}\in \mathbb Z$.   We have chosen to use the simpler stability segment as we consider it more natural.
\end{remark}

Our next task is to construct the necessary twisting for the stability segment $\zeta(\cdot)$.

\begin{lemma}\label{lem:existencetwistsI}
Let $\lambda,\mu\in \mathbb R_{>0}$ and $A$ be a line bundle.  
Then, there exists a formal sum $B$ of line bundles with positive coefficients such that $\rank(B) = \lambda$, $c_1(B)=0$ and $\ch_2(B) = \mu c_1(A)^2$.
\end{lemma}
\begin{proof}
Let $\alpha := \mu n^{-2}$, where $n$ is an integer large enough so $\alpha<\lambda/2$.  Then, 
$$B:= \alpha (A^n + A^{-n}) + (\lambda - 2\alpha) \mathcal O_X$$ has the desired properties.
\end{proof}

\begin{corollary}\label{cor:existencetwistsII}
Let $L_0$ and $L_1$ be line bundles and $\alpha_0,\alpha_1\in \mathbb R_{>0}$.  Then, there exists a formal sum $C$ of line bundles with positive coefficients such that $\rank(C) =1$, $c_1(C)=0$, and 
$$\ch_2(C) = \sum\nolimits_{k=0}^1 \alpha_k c_1(L_k)^2.$$
\end{corollary}
\begin{proof}
Apply the previous lemma twice to get $C_k$ with $\rank(C_k)=1/2$, $c_1(C_k)=0$, and $\ch_2(C_k) = \alpha_k c_1(L_k)^2$.  Then, set $C:= C_1 + C_2$.  
\end{proof}

\begin{lemma}\label{lem:alphabeta}
Let $q_{ji}\in \mathbb R$ for $j,i=0,1$ and $r_0,r_1\in \mathbb R_{>0}$.  Suppose also $c_{ji}\in \mathbb R$ with $c_{0k}\neq c_{1k}$ for $k\in \{0,1\}$.     Then, for $\lambda\gg 0$ and all $b>0$, there exist $\alpha_{jik}\in \mathbb R_{>0}$ such that
\begin{align}
\sum\nolimits_j r_j \alpha_{jik} &= r_k b(\lambda + q_{ki})&&\text{ for }i,k=0,1\label{eq:alphabeta1}\\
r_0 c_{0k} (\alpha_{01k} - \alpha_{00k}) &= r_1 c_{1k} (\alpha_{10k} - \alpha_{11k})&&\text{ for } k=0,1.\label{eq:alphabeta2}
\end{align}
\end{lemma}
\begin{proof}
The equations for $k=0$ and $k=1$ separate, so it is sufficient to prove the claim for $k=0$.  By linear trivialities, since $c_{00}\neq c_{10}$,  there certainly exist $\beta_{ji}\in \mathbb R$ with
\begin{align*}
r_0 \beta_{00} + r_1 \beta_{10} &= r_0  q_{00}\\
r_0 \beta_{01} + r_1 \beta_{11} &=  r_0   q_{01}\\
r_0 c_{00} (\beta_{01} - \beta_{00}) &= c_{10} r_1 (\beta_{10} - \beta_{11}).
\end{align*}  
Now, let $\alpha_{0i0} = b (\beta_{0i} +  \lambda/2) $ and $\alpha_{1i0} =b( \beta_{1i} +  \lambda r_0/2r_1)$ for $i=0,1$ which are positive as long as $\lambda$ is sufficiently large.
\end{proof}

\begin{proof}[Proof of Lemma \ref{lem:finaltwist}]
Set $r_j := \rank (B_j(\overline{s}))= \sigma_j(\overline{t})$ and
$$q_{ji} := \frac{a(s_i-\overline{s})(\sigma_j(t_1)- \sigma_j(t_0))}{\sigma_j(\overline{t})}.$$
Then, recall we are to find $D_j(r)$ such that
\begin{enumerate}
\item $\rank(D_j(r)) = r_j$ and $c_1(D_j(r)) = c_1(B_j(\overline{s})).$  
\item For $i=0,1$, we have
$$\sum\nolimits_{j} \ch_2(D_j(i)) = b \sum\nolimits_{j} r_j (\lambda + q_{ji})c_1(L_j)^2 +\sum\nolimits_j \ch_2(B_j(\overline{s})). $$
\item The quantity $\sum\nolimits_{j} c_1(L_j) \ch_2(D_j(r))$ is independent of $r$.
\end{enumerate}
To do this, let $c_{ik} = \int_X c_1(L_i) c_1(L_k)^2$.  
Then, our assumption \eqref{eq:gen} says that $c_{0k}\neq c_{1k}$ for $k=0,1$.   
So, we can apply Lemma \ref{lem:alphabeta} to get $\alpha_{jik}>0$ for $i,j,k\in \{0,1\}$ 
that satisfy \eqref{eq:alphabeta1} and \eqref{eq:alphabeta2}. From Corollary \ref{cor:existencetwistsII} there exist formal sums of line bundles $C_{ji}$ 
with positive coefficients, such that $\rank(C_{ji})=1$ and $c_1(C_{ji})=0$ and
\begin{equation}
\ch_2(C_{ji}) = \alpha_{ji0} c_1(L_0)^2 + \alpha_{ji1} c_1(L_1)^2 = \sum\nolimits_k \alpha_{jik} c_1(L_k)^2.\label{ch2Cji}
\end{equation}
Setting
$$ D_j(r) := B_j(\overline{s}) \otimes ((1-r) C_{j0} + r C_{j1}),$$
 one sees immediately that property (1) holds.  For (2), we compute
\begin{align*}
 \sum\nolimits_j \ch_2(D_j(0)) &= \sum\nolimits _j \ch_2(B_j(\overline{s}) \otimes C_{j0}) \\
&=\sum\nolimits_j r_j \ch_2(C_{j0}) + \ch_2(B_j(\overline{s})) \\
&=\sum\nolimits_j r_j (\alpha_{j00} c_1(L_0)^2 + \alpha_{j01} c_1(L_1)^2) + \sum\nolimits_j \ch_2(B_j(\overline{s}) .
\end{align*}
With \eqref{eq:alphabeta1} this yields property (2) when $i=0$ (and the computation for $i=1$ is the same).  Finally,
\begin{align*}
\sum\nolimits_j c_1(L_j) \ch_2(D_j(r)) &= \sum\nolimits_j c_1(L_j)  . \left(\ch_2(B_j(\overline{s})) + r_j ( (1-r) \ch_2(C_{j0}) + r \ch_2(C_{j1}) )\right)\\
&= r \sum\nolimits_j r_j c_1(L_j).(\ch_2(C_{j1}) - \ch_2(C_{j0})) + \epsilon,
\end{align*}
where $\epsilon$ is independent of $r$.  But using \eqref{ch2Cji} and \eqref{eq:alphabeta1},
\begin{align*}
\sum\nolimits_j r_j c_1(L_j).(\ch_2(C_{j1}) - \ch_2(C_{j0})) &= \sum\nolimits_k \sum\nolimits_j r_j (\alpha_{j1k}  - \alpha_{j0k}) c_{jk}=0,
\end{align*}
which gives property (3).
\end{proof}

\begin{proof}[Proof of Lemma \ref{lem:step3endpoints}]
We have to show that  (1) $\mathcal S_K$ detects the change in multi-Gieseker stability for $\zeta(\cdot)$, (2) $\zeta(\cdot)$ is a bounded and uniform segment of stability parameters and (3) for $b\in \R_{>0}$ sufficiently large, $\zeta(0) \equiv \eta(s_0)$ and $\zeta(1) \equiv \eta(s_1)$.

We recall that by construction $ \zeta(r) =  (L_0,L_1; D_0(r), D_1(r))$ for $r\in [0,1]$, where
\begin{align}\label{eq:Dj}
\rank (D_j(r)) &=\rank (B_j(\overline{s})) =r_j \text{ and } c_1(D_j(r)) = c_1(B_j(\overline{s})) 
\end{align}
In particular,
$$ \vol(L_0) \rank (D_0(r)) + \vol(L_1) \rank (D_1(r)) = \vol(L_0) \rank (B_0(\overline{s})) + \vol(L_1) \rank (B_1(\overline{s})) = 1,$$
so $\zeta(r)$ is in fact a stability segment.  

Clearly, $\zeta(\cdot)$ is bounded as any sheaf that is semistable with respect to $\zeta(r)$ for some $r\in [0,1]$ is slope semistable with respect to $\gamma_{{\overline{t}}}$.  Moreover, from this (1) follows in the same way as in the proof of Lemma \ref{lem:step2open}.  We next show that $\zeta(\cdot)$ is uniform.  To see this, let $E$ be torsion-free, so by Riemann-Roch
$$ P^{\zeta(r)}_E(k) = \rank(E) k^3 + a_1 k^2 + a_2 k + a_3$$
with 
\begin{align*}
a_1 &= \sum\nolimits_j c_1(L_j)^2.(r_j\Hilb_1(E) +\rank(E) c_1(D_j(r))\\
a_2 &= \sum\nolimits_j c_1(L_j). \bigl[ r_j\Hilb_2(E) + \rank(E)\ch_2(D_j(r)+c_1(E)c_1(D_j(r))\\ & \quad \quad\quad \quad \quad\quad \quad\quad+\rank(E)c_1(D_j(r))\Todd_1(X)\bigr].
\end{align*}
Now, by \eqref{eq:Dj} we see $r_j = \sigma_k(\overline{t})$ and $c_1(D_j(r))$ are independent of $r$, and so $a_1$ is independent of $r$.  Moreover, by Lemma \ref{lem:finaltwist}(3) the quantity $\sum_j c_1(L_j)\ch_2(D_j(r))$ is also independent of $r$, and hence so is $a_2$.  Thus, the only term that depends non-trivially on $r$ is $a_3$, and thus $\zeta(\cdot)$ is uniform, proving (2).

We now prove (3) by showing $\zeta(0)\equiv \eta(s_0)$ for $b \gg 0$ (the other endpoint being similar).  Exactly as in the claim at the start of the proof of Lemma \ref{lem:step2endpoints}, it is sufficient to prove that if $E$ is torsion-free of topological type $\tau$, and $F\subset E$ is saturated with $F\in \mathcal S_K$ it holds that
\begin{equation}
p_{F\subset E}^{\eta(s_0)} (\le) 0 \text{ if and only if } p_{F\subset E}^{\zeta(0)} (\le) 0.\label{eq:lem:step3endpointssufficient}
\end{equation}

To prove this, with $u_i(s)$ defined as in \eqref{eq:ui} one computes using Riemann-Roch \eqref{eq:h1} that if $F\subset E$, then
\begin{equation}
 p_{F\subset E}^{\zeta(r
)} = \langle \langle u_1 || u_2(\overline{s}) || u_3(\overline{s}) + \delta(r) \rangle \rangle,\label{eq:reducedzetar}
\end{equation}
where
$$ \delta(r) = \sum\nolimits_j \rank(B_j(\overline{s})) \hilb_1(F,E).[ (1-r)\ch_2(C_{j0}) + r\ch_2(C_{j1})].$$
After some manipulation with \eqref{eq:ui} and \eqref{ch2Cji}, we in fact have
\begin{equation}
\delta(r) = b[\lambda u_1 + (1-r) u_2(s_0) + ru_2(s_1) - u_2(\overline{s})].\label{eq:reducedzetaruseful}
\end{equation}
Thus, in particular
$$ p_{F\subset E}^{\zeta(0)} = \langle \langle u_1 || u_2(\overline{s}) || u_3(\overline{s}) + b(u_2(s_0) - u_2(\overline{s}) - \lambda u_1) \rangle \rangle.$$
Now, clearly $p_{F\subset E}^{\zeta(0)}(\le) 0$ if and only if
$$\langle \langle u_1 || u_2(\overline{s}) || u_3(\overline{s}) + b(u_2(s_0) -u_2(\overline{s})\rangle \rangle(\le) 0,$$
which in turn occurs if and only if
\begin{equation}\label{lem:zetaend2}
\langle \langle u_1 || u_2(\overline{s}) || u_3(\overline{s}) + bu_2(s_0)\rangle \rangle(\le) 0.
\end{equation}
\smallskip

\noindent\emph{Final Claim: }For $b$ sufficiently large, the inequality \eqref{lem:zetaend2} holds if and only if
$$\langle \langle u_1 || u_2(s_0) || u_3(s_0) \rangle \rangle(\le) 0.$$
\smallskip

The argument for this is similar to what has been done before, cf.~the proof of Lemma~\ref{lem:step2endpoints}.  Assume \eqref{lem:zetaend2} holds.  Then, $u_1\le 0$.  If this is strict we are done.  So, we may assume $u_1=0$, and thus $u_2(\overline{s})\le 0$.  If this inequality is strict, then as $s_0$ is in an adjacent chamber to $\overline{s}$ we have $u_2(s_0)<0$ and we are done.  Otherwise $u_2(\overline{s})=0$ and $u_3(\overline{s}) + bu_2(s_0)\le 0$.  As $b$ is large, this implies $u_2(s_0)\le 0$.  Again, if this is strict, we are done.  Otherwise 
\begin{equation}\label{lem:zetaend1}
u_2(s_0)=0 \text{ and } u_3(\overline{s})(\le) 0.
\end{equation}
By Lemma \ref{problem} below, the first of these statements gives $\hilb_1(F,E)c_1(L_j)^2=0$ and also $\hilb_2(F,E)c_1(L_j)=0$ for $j=0,1$.    Hence, from \eqref{eq:ui}, \eqref{eq:h1}, and \eqref{eq:ranksandcherns} we actually have $u_3(s) = h_3(\overline{t})$, which is independent of $s$.  Consequently, from the second statement in \eqref{lem:zetaend1} we deduce $u_3(s_0)(\le) 0$ as well, proving one direction of the claim.

For the converse, assume $\langle \langle u_1 || u_2(s_0) || u_3(s_0) \rangle \rangle\le 0$.  Then, $u_1\le 0$, and if strict inequality holds we are done.  Otherwise, $u_1=0$ and $u_2(s_0)\le 0$, and so $u_2(\overline{s})\le 0$ by continuity.   If strict inequality holds, then $u_2(s_0)<0$ as well, since $s_0$ is in an adjacent chamber, and we are done.  So, we may assume $u_2(\overline{s})=0$.  We divide into two cases.  In the first case $u_2(s_0)=0$.  Then, $u_3(s_0) (\le) 0$ and hence by continuity $u_3(\overline{s})(\le) 0$ (the case of equality here uses that $s_0$ lies in an open chamber, so $u_3(s_0)=0$ implies $u_3(\cdot)\equiv 0$).   In the second case $u_2(s_0)<0$.  But for $b$ large this implies $u_3(\overline{s}) + bu_2(s_0)<0$, and we are done here as well.  This proves the Final Claim, and hence concludes the proof of Lemma~\ref{lem:step3endpoints}.
\end{proof}

\begin{lemma}\label{problem}
Let $E$ be of topological type $\tau$ and $F\in \mathcal S_{K}$.  With $u_{i}$ defined as in \eqref{eq:ui}, suppose that $u_1=0$ and $u_2(s_0)=0$.  
Then, $u_2\equiv 0$, $\hilb_1(F,E)c_1(L_j)^2=0$, and $$\hilb_2(F,E)c_1(L_j)=0$$ for $j=0,1$.  
\end{lemma}
\begin{proof}
 As $s_0$ is in an open chamber, the second assumption implies $u_2 \equiv 0$. 
But looking back at \eqref{eq:ui} this implies $h_1\equiv 0$, which gives $\hilb_1(F,E)c_1(L_j)^2=0$ for $j=0,1$ by \eqref{eq:hit}.    Looking again at \eqref{eq:ui} the fact that $u_2\equiv 0$ now implies $h_2(\overline{t})=0$.  But by the assumption that $L_0$ and $L_1$ are separated by a single wall of the first kind (and not by any walls of the second kind) this is impossible unless $\hilb_2(F,E)c_1(L_j)=0$ for $j=0,1$ as well. 
\end{proof}
\begin{remark}
We mention again that it is the proof of Lemma~\ref{problem} where the assumption on the two line bundles $L_0$ and $L_1$ being separated by a single wall of the first kind is used.
\end{remark}

\begin{proof}[Proof of Lemma \ref{lem:step3open}]
We end this ordeal by proving that for $b$ sufficiently large $\zeta(\cdot)$ is open.  As $\mathcal{S}_K$ detects the change of stability for $\zeta(\cdot)$, by Lemma \ref{lem:detectchangemultigieseker} it is sufficient to prove that for $r$ sufficiently small, if $E$ is torsion-free of type $\tau$ and $F\subset E$ is saturated with $F\in \mathcal S_K$, then
$$p_{F\subset E}^{\zeta(0)} (\le) 0 \text{ if and only if } p_{F\subset E}^{\zeta(r)} (\le) 0$$
and
$$p_{F\subset E}^{\zeta(1)} (\le) 0 \text{ if and only if } p_{F\subset E}^{\zeta(1-r)} (\le) 0.$$
We will prove the first of these, the second being proved in precisely the same way.   So let $E$ and $F$ be such sheaves.  Recall from \eqref{eq:reducedzetar}, \eqref{eq:reducedzetaruseful} that
\begin{equation*}
 p_{F\subset E}^{\zeta(r
)} = \langle \langle u_1 || u_2(\overline{s}) || u_3(\overline{s}) + \delta(r) \rangle \rangle
\end{equation*}
with
\begin{equation*}
\delta(r) = b[\lambda u_1 + (1-r) u_2(s_0) + ru_2(s_1) - u_2(\overline{s})].
\end{equation*}
Since $u_1$ and $u_2(\overline{s})$ are independent of $r$, it is sufficient to show the following: if $u_1=u_2(\overline{s})=0$ and $u_3(\overline{s}) + \delta(0)=0$, then $u_3(\overline{s}) +\delta(\cdot)\equiv 0$.  

So, assume  $u_1=u_2(\overline{s})=0=u_3(\overline{s}) + \delta(0)$. Then, the first equality implies $\delta(0) = b u_{2}(s_{0})$, and so
$$ 0 = u_{3}(\overline{s}) + b u_{2}(s_{0}).$$
Since $b$ is arbitrarily large, this implies $u_{2}(s_{0})=0$.    Hence, the linear function $s\mapsto u_{2}(s)$ vanishes at the two distinct points $s=\overline{s}$ and $s=s_{0}$, and thus $u_{2}(s)=0$ for all $s\in [0,1]$.  In particular, $u_{2}(s_{1})=0$, and so the expression $u_{3}(\overline{s}) + \delta(r)$ is actually independent of $r$, which completes the proof.
\end{proof}

\begin{remark}
As $\zeta(\cdot)$ is uniform, not only is it almost perfect but in fact is perfect (that is, the moduli spaces that appear on the walls with respect to $\zeta(r)$ for $r\in [0,1]$ are also related to the moduli spaces of the endpoints by Thaddeus-flips through moduli spaces of sheaves).  With notation above, choose $\tilde{r}\in (0,1)$ so $(1-\tilde{r}) s_0 + \tilde{r} s_1 = \overline{s}$.  Then using linearity of $u_2$ one sees easily from \eqref{eq:reducedzetaruseful} that $\eta(\overline{s}) \equiv \zeta(\tilde{r})$.  Thus $\eta(\cdot)$ is also perfect.  It does not seem to be the case that the segment $\sigma(\cdot)$ is also perfect; the same approach fails due to the additional $\epsilon(s)$ term appearing in $u_3(\cdot)$ \eqref{eq:ui}.  However if one instead uses the twisting as mentioned in Remark \ref{rmk:simplified} this $\epsilon(\cdot)$ term vanishes, and one can show that if $(1-\tilde{s})t_0 + \tilde{s} t_1 = \overline{t}$ then $\eta(\tilde{s}) \equiv \sigma(\overline{t})$, and so this new $\sigma(\cdot)$ is also perfect.
\end{remark}

\begin{remark}
  It seems likely that the above ideas can be extended to prove that, on a threefold at least, if $\underline{L}= (L_1,\ldots,L_{j_0})$ is fixed, then for any points $\sigma,\sigma'\in (\mathbb Q_+)^{j_0}$ we have $\mathcal M_{\sigma} \leftrightdash \mathcal M_{\sigma'}$.  If this were the case, then using the openness part of \Iref{Proposition}{prop:convex}, one concludes that for any two $\omega_{1},\omega_{2}\in \Amp(X)_{\R}$ the (projective) moduli spaces $\mathcal M_{\omega_1}$ and $\mathcal M_{\omega_2}$ are related by Thaddeus-flips through moduli spaces of sheaves.   
\end{remark}

\vspace{0.4cm}

\addtocontents{toc}{\protect\setcounter{tocdepth}{0}}

\vspace{1cm}


\begin{thebibliography}{MFK94}
\addtocontents{toc}{\protect\setcounter{tocdepth}{1}}
\vspace{0.3cm}

\bibitem[Ber14]{Bertramb} Aaron Bertram, \emph{Stability and Positivity}, talk given as part of ``Positivity of Linear Series and Vector Bundles'', workshop at Banff International Research Station, Jan.~2014, available online at \url{www.birs.ca/events/2014/5-day-workshops/14w5056/videos}

\bibitem[Con09]{Consul} Luis Alvarez-Consul
\emph{Some results on the moduli spaces of quiver bundles} Geom. Dedicata 139 (2009), 99--120. 


\bibitem[DH98]{DolgachevHu} 
Igor Dolgachev and Yi Hu, \emph{Variation of geometric invariant theory quotients}, Inst. Hautes \'Etudes Sci. Publ. Math. \textbf{87} (1998), 5--56. 

\bibitem[Ful98]{F} William Fulton,  \emph{Intersection theory}, Springer-Verlag, Berlin, 1998.

\bibitem[GRT]{GRTI}
Daniel Greb, Julius Ross, and Matei Toma,\emph{Variation of Gieseker Moduli Spaces via Quiver GIT, I}, preprint (To appear in Geometry and Topology), \texttt{arXiv:1409.7564}, 2014.

\bibitem[GT13]{GrebToma}
Daniel Greb and Matei Toma, \emph{Compact moduli spaces for slope-semistable sheaves}, preprint (To appear in Geometry and Topology), \texttt{arXiv:1303.2480}, 2013.

\bibitem[Gro95]{Groth}
Alexander Grothendieck, \emph{Techniques de construction et th\'eor\`emes d'existence en g\'eom\'etrie alg\'ebrique. IV. Les sch\'emas de Hilbert}, S\'em. Bourbaki, Vol. 6, Exp. No.~221, Soc.~Math.~France, Paris, 1995, pp.~249--276.

 \bibitem[HL10]{Bible}
 Daniel Huybrechts and Manfred Lehn, \emph{The geometry of moduli spaces of
   sheaves}, second ed., Cambridge Mathematical Library, Cambridge University
   Press, Cambridge, 2010.

\bibitem[JII]{JoyceII} Dominic Joyce \emph{Configurations in abelian categories. II. Ringel-Hall algebras}, Adv. Math. \textbf{210} (2007), no. 2, 635--706. 

\bibitem[Lie07]{Lieblich} Max Lieblich, \emph{Moduli of twisted sheaves},     Duke Math. J. \textbf{138} (2007), no.~1, 1--178.

\bibitem[MW97]{MatsukiWentworth}
Kenji Matsuki and Richard Wentworth, \emph{Mumford-{T}haddeus principle on the
  moduli space of vector bundles on an algebraic surface}, Internat. J. Math.
  \textbf{8} (1997), no.~1, 97--148.

\bibitem[Sch00]{Schmitt}
Alexander Schmitt, \emph{Walls for {G}ieseker semistability and the
  {M}umford-{T}haddeus principle for moduli spaces of sheaves over higher
  dimensional bases}, Comment. Math. Helv. \textbf{75} (2000), no.~2, 216--231.

\bibitem[Tha96]{Thaddeus} Michael Thaddeus, \emph{Geometric invariant theory and flips},
J. Amer. Math. Soc. \textbf{9} (1996), no. 3, 691–723. 

\end{thebibliography}
\end{document}